\documentclass[12pt,reqno]{amsart}

\usepackage{amssymb}
\usepackage[latin1]{inputenc}
\usepackage{dsfont}
\usepackage{hyperref}
\usepackage{fullpage}

\newtheorem{theorem}{Theorem}[section]
\newtheorem{lemma}[theorem]{Lemma}
\newtheorem{proposition}[theorem]{Proposition}
\newtheorem{corollary}[theorem]{Corollary}
\theoremstyle{definition}

\theoremstyle{remark}
\newtheorem{remark}[theorem]{Remark}

\numberwithin{equation}{section}

\newcommand{\hg}{{}_2F_1}
\newcommand{\arccot}{\textrm{arccot}}

\begin{document}


\title[Eigenvalue Probabilities and Cauchy Weights]{Scaled limit and rate of convergence for the largest eigenvalue from the generalized Cauchy random matrix ensemble}


%    author one information % \author[short version for running head]{name for top of paper}

\author[J. Najnudel]{Joseph Najnudel}
\address{Institut f\"ur Mathematik, Universit\"at Z\"urich, Winterthurerstrasse 190,
8057-Z\"urich, Switzerland}
\email{\href{mailto:joseph.najnudel@math.uzh.ch}{joseph.najnudel@math.uzh.ch}}
%\thanks{}

%    author two information

\author[A. Nikeghbali]{Ashkan Nikeghbali}
\email{\href{mailto:ashkan.nikeghbali@math.uzh.ch}{ashkan.nikeghbali@math.uzh.ch}}

\author[F. Rubin]{Felix Rubin}\email{\href{mailto:felix.rubin@math.uzh.ch}{felix.rubin@math.uzh.ch}}

\thanks{All authors are supported by the Swiss National Science Foundation (SNF) grant 200021\_119970/1}

%    \subjclass is required. %\subjclass[2000]{Primary }

\date{\today}

\begin{abstract}
In this paper, we are interested in the asymptotic properties for the largest eigenvalue of the Hermitian random matrix ensemble, called the Generalized Cauchy ensemble $GCyE$, whose eigenvalues PDF is given by $$\textrm{const}\cdot\prod_{1\leq j<k\leq N}(x_j-x_k)^2\prod_{j=1}^N
(1+ix_j)^{-s-N}(1-ix_j)^{-\overline{s}-N}dx_j,$$where $s$ is a complex number such that $\Re(s)>-1/2$ and where $N$ is the size of the matrix ensemble. Using results by Borodin and Olshanski \cite{Borodin-Olshanski}, we first prove that for  this ensemble, the law of the largest eigenvalue divided by $N$ converges to some probability distribution for all $s$ such that $\Re(s)>-1/2$.  Using results by Forrester and Witte \cite{Forrester-Witte2} on the distribution of the largest eigenvalue for fixed $N$, we also express the limiting probability distribution in terms of some non-linear second order differential equation. Eventually, we show that the convergence of the probability distribution function of the re-scaled largest eigenvalue to the limiting one is at least of order  $(1/N)$.

\end{abstract}

\maketitle
%\tableofcontents
\section{Introduction and results}\label{intro}

Let $H(N)$ be the set of Hermitian matrices endowed with  the measure
\begin{equation}\label{Haar measure on H}
\textrm{const}\cdot \det(1+X^2)^{-N}\prod_{1\leq j<k\leq
N}dX_{jk}\prod_{i=1}^N dX_{ii}, \qquad X\in H(N),
\end{equation}
where const is a normalizing constant, such that the total mass of $H(N)$ is equal to one. This measure is the analogue of the normalized Haar measure $\mu_N$ on the unitary group $U(N)$, if one relates $U(N)$ and $H(N)$ via the Cayley transform: $H(N)\ni X\mapsto U=\frac{X+i}{X-i}\in U(N)$. The measure \eqref{Haar measure on H} can be deformed to obtain the following two parameters  probability measure:
\begin{equation}\label{Full Hua-Pickrell}
\textrm{const}\cdot\det((1+iX)^{-s-N})\det((1-iX)^{-\overline{s}-N})\prod_{1\leq
j<k\leq N}dX_{jk}\prod_{i=1}^N dX_{ii},
\end{equation}
where $s$ is a complex parameter such that $\Re{s}>-1/2$ (otherwise the quantity involved in (\ref{Full Hua-Pickrell}) does not integrate as is proved in \cite{Borodin-Olshanski}). Following Forrester and Witte \cite{Forrester-Witte2} and \cite{Forrester-Witte}, we call this measure the \textit{generalized Cauchy measure} on $H(N)$. The name is chosen because if $s=0$ and $N=1$, \eqref{Full Hua-Pickrell} is nothing else than the density of a Cauchy random variable.  This measure can be projected onto the space $\mathds{R}^N/S(N)$, where $S(N)$ is the symmetric group of order $N$ and the quotient space is considered as the space of all (unordered) sets of eigenvalues of matrices in $H(N)$. This projection gives the eigenvalue density
\begin{equation}\label{Hua-Pickrell on H}
\textrm{const}\cdot\prod_{1\leq j<k\leq N}(x_j-x_k)^2\prod_{j=1}^N
w_H(x_j)dx_j,\end{equation}
where $w_H(x_j)=(1+ix_j)^{-s-N}(1-ix_j)^{-\overline{s}-N}$, and where the $x_j$'s denote the  eigenvalues, and as usual, the constant is  chosen so that the total mass of $\mathds{R}^N/S(N)$ is equal to one. $H(N)$, endowed with the generalized Cauchy measure, shall be called the \textit{generalized Cauchy random matrix ensemble}, noted $GCyE$.

If one replaces the weight $w_H(x)$ in (\ref{Hua-Pickrell on H}) by $w_2(x)=e^{-x^2}$, then one obtains the Hermite ensemble. Similarly, the choice  $w_L(x)=x^a e^{-x}$ on $\mathbb{R}_+$ or $w_J(x)=(1-x)^\alpha (1+x)^\beta$ for $-1\leq x\leq 1$ leads to the Laguerre or Jacobi ensemble. The three classical weight functions $w_2$, $w_L$ and $w_J$ occur in the eigenvalue PDF for certain ensembles of Hermitian matrices based on matrices with independent Gaussian entries (see for example \cite{Forrester}). In \cite{Adler-Forrester}, the defining property of a classical weight function in this context was identified as the following fact: If one writes the weight function $w(x)$ of an ensemble as $w(x)=e^{-2V(x)}$, with $2V'(x)=g(x)/f(x)$, $f(x)$ and $g(x)$ being polynomials in $x$, then the operator $n=f(d/dx)+(f'-g)/2$ increases the degree of the polynomials by one, and thus, $\textrm{deg}f\leq2$, and $\textrm{deg}g\leq1$. If $s\in(-1/2,\infty)$, this property actually also holds for the $GCyE$, and we obtain a fourth classical weight function (see also Witte and Forrester \cite{Forrester-Witte}).  However, the construction of the matrix model for the $GCyE$ is different from the construction of the other three classical ensembles: A matrix model for the $GCyE$ will not have independent entries, but one can construct the ensemble via the Cayley transform. Indeed, following Borodin and Olshanski \cite{Borodin-Olshanski} (see also \cite{Forrester, Forrester-Witte2, Forrester-Witte}) the measure \eqref{Full Hua-Pickrell} is, via the Cayley-transform, equivalent to the deformed normalized Haar measure $\textrm{const}\cdot\det((1-U)^{\overline{s}})\det((1-U^*)^s)\mu_N(dU)$, $U\in U(N)$. If we denote by $e^{i\theta_j}$, $j=1,\ldots,N$, the eigenvalues of a unitary matrix with $\theta_j\in[-\pi,\pi]$, the deformed Haar measure can, as in the Hermitian case, be projected to the eigenvalue probability measure to obtain the PDF
\begin{equation}\label{Hua-Pickrell on U}
\textrm{const}\cdot\prod_{1\leq j<k\leq N}
|e^{i\theta_j}-e^{i\theta_k}|^2\prod_{j=1}^N w_U(\theta_j)
d\theta_j,
\end{equation}
where $w_U(\theta_j)=(1-e^{i\theta_j})^{\overline{s}}(1-e^{-i\theta_j})^s,$ and $\theta_j\in[-\pi,\pi]$. This measure is defined on $\mathds{S}^N/S(N)$, where $\mathds{S}$ is the complex unit circle. Note, that this eigenvalue measure has a singularity at $\theta=0$, if $s\neq0$. Borodin and Olshanski \cite{Borodin-Olshanski} studied the measures \eqref{Full Hua-Pickrell}, \eqref{Hua-Pickrell on H} and \eqref{Hua-Pickrell on U} in great detail due to their connections with representation theory of the infinite dimensional unitary group $U(\infty)$.

When $s\in(-1/2,\infty)$, \eqref{Hua-Pickrell on U} is nothing else than the eigenvalue distribution of the circular Jacobi unitary ensemble. This is a generalization of the Circular Unitary ensemble corresponding to the case $s=0$. In fact, if $s=1$, this corresponds to the CUE case with one eigenvalue fixed at one. More generally, for $s\in(-1/2,\infty)$ the singularity at one corresponds, in the log-gas picture, to a impurity with variable charge fixed at one, and mobile unit charges represented by the eigenvalues (see Witte and Forrester \cite{Forrester-Witte}, and also \cite{Forrester-Witte3}). It is the singularity at one that makes the study of this ensemble more difficult than the CUE. In the special case when $s=0$, one can obtain the eigenvalues with PDF (\ref{Hua-Pickrell on H}) from the eigenvalues of the circular unitary ensemble using a stereographic projection (see the book of Forrester \cite{Forrester}, Chapter 2, Section 5 on the Cauchy ensemble). In fact, in this case, we get that \eqref{Hua-Pickrell on H} represents the Boltzmann factor for a one-component log-gas on the real line subject to the potential $2V(x)=N\log(1+x^2)$. This corresponds to an external charge of strength $-N$ placed at the point $(0,1)$ in the plane (this can also be generalized to arbitrary inverse temperature $\beta$ as given in the previous reference). Moreover, note that when $s\neq0$, a construction of a random matrix ensemble with eigenvalue PDF \eqref{Hua-Pickrell on U} is given in \cite{Bourgade-Nikeghbali-Rouault2}.

In this paper, we are interested in the convergence and the asymptotic distribution of the re-scaled largest eigenvalue of a random matrix drawn from the generalized Cauchy ensemble, \textit{for all admissible values of the parameter $s$}, namely $\Re(s)>-1/2$.  Moreover, we will also address the problem of the rate of convergence in such a limit Theorem. In random matrix theory, the distribution of the largest eigenvalue as well as the problem of the convergence of the scaled largest eigenvalue, have received much attention (see e.g. \cite{peche}, \cite{soshnikov2}, \cite{soshnikov}, \cite{Tracy-Widom2}). Also the latter problem on the rate of convergence has been studied, especially in \cite{Garoni-Forrester-Frankel} and  \cite{Choup} for GUE and LUE matrices, and in \cite{Johnstone} as well as in \cite{El-Karouie} for Wishart matrices. %Regarding the question of convergence and asymptotic distribution, it is well known that in the example of the GUE, Tracy and Widom proved in \cite{Tracy-Widom2} that the law of the appropriately re-scaled largest eigenvalue converges to the so called Tracy-Widom distribution which can be characterized in terms of a Painlev? II equation. This result has been extended to more general ensembles by Tracy and Widom \cite{Tracy-Widom}, Soshnikov \cite{soshnikov2} and later by other authors (see \cite{soshnikov} for an overview).
To deal with the law of the largest eigenvalue, there is a well established methodology (see \cite{Mehta}) for matrix ensembles with eigenvalue PDF of the form
\begin{equation}\label{Hua on H}
\textrm{const}\cdot\prod_{1\leq j<k\leq N}(x_j-x_k)^2\prod_{j=1}^N
w(x_j)dx_j,\end{equation}
where $w(x)$ is a weight function on $\mathbb{R}$.
If one can define the set of monic orthogonal polynomials $\{p_n\}$ with respect to the  weight function $w(x)$ on $\mathds{R}$, then one defines the integral operator $K_N$ on $L_2(\mathds{R})$, associated with the kernel $K_N(x,y):=\sum_{i=0}^{N-1}\frac{p_i(x)p_i(y)}{\|p_i\|^2} \sqrt{w(x) w(y)}$. Using this kernel, the formula to describe probabilities of the form
\begin{equation*}
E(k,J):=\mathbb{P}[\textrm{there are exactly }k\textrm{ eigenvalues inside the interval } J],
\end{equation*}
where $J\subset\mathds{R}$ and $k\in\mathds{N}$, is (see \cite{Mehta}):
\begin{equation*}
E(k,J)=\frac{(-1)^k}{k!}\frac{d^k}{dx^k}\det(I-xK_N)|_{x=1},
\end{equation*}
where the determinant is a Fredholm determinant and the operator $K_N$ is restricted to $J$. Note that if one takes $J=(t,\infty)$ for some $t\in\mathds{R}$, then $E(0,(t,\infty))$ is simply the probability distribution  of the largest eigenvalue, denoted from now on by $\lambda_1(N)$, of a $N\times N$ matrix in the respective ensemble. In their pioneering work \cite{Tracy-Widom}, Tracy and Widom give a system of completely integrable differential equations to show how the probability $E(0,J)$ can be linked to solutions of certain Painlev? differential equations. Tracy and Widom apply their method to the finite Hermite, Laguerre and Jacobi ensembles. Moreover, one can also apply the method to scaling limits of random matrix ensembles, when the dimension $N$ goes to infinity. The sine kernel and its Painlev?-V representation for instance, as obtained by Jimbo, Miwa, M?ri and Sato \cite{JMMS}, arise if one takes the scaling limit in the bulk of the spectrum of the Gaussian Unitary Ensemble and of many other Hermitian matrix ensembles (see e.g. \cite{Kamien-Politzer}, \cite{Mahoux-Mehta}, \cite{Nagao-Wadati} and \cite{Pastur}). On the other hand, if one scales appropriately at the edge of the Gaussian Unitary Ensemble, one obtains an Airy kernel in the scaling limit with a Painlev?-II representation for the distribution of the largest eigenvalue (see \cite{Tracy-Widom2}). Similar results have been obtained for the edge scalings of the Laguerre and Jacobi ensembles, where the Airy kernel has to be replaced by the Bessel kernel and the Painlev?-II equation by a Painlev?-V equation (see Tracy and Widom \cite{Tracy-Widom3}). Soshnikov \cite{soshnikov} gives an overview on scaling limit results for large random matrix ensembles.

For the eigenvalue measure \eqref{Hua-Pickrell on H}, Borodin and Olshanski \cite{Borodin-Olshanski} give the kernel in the finite $N$ case, denoted by $K_N$ in the following (see Theorem \ref{correlation Kernel}), as well as a scaling limit of this kernel, when $N\rightarrow\infty$, denoted by $K_\infty$ (see \eqref{definition of K infinity}). Using the kernel $K_N$, one can set up the system of differential equations in the way of Tracy and Widom for the law $E(0,(t,\infty))$ of the largest eigenvalue $\lambda_1(N)$, for any $t\in\mathds{R}$. In the case of a real parameter $s$, this has been done by Forrester and Witte in \cite{Forrester-Witte}. They obtain a characterization of the law of the largest eigenvalue in terms of a Painlev?-VI equation. More precisely, $(1+t^2)$ times the logarithmic derivative of $E(0,(t,\infty))$ satisfies a Painlev?-VI equation. The same method suitably modified leads to a generalization of this result for  complex $s$. However, the method of Tracy and Widom has the drawback that it only works for $s$ with $\Re{s}>1/2$. Forrester and Witte propose in \cite{Forrester-Witte2} an alternative  method which makes use of $\tau$-function theory, to derive the Painlev?-VI characterization for $E(0,(t,\infty))$ for any $s$ such that $\Re{s}>-1/2$.

To sum up, for the generalized Cauchy ensemble, it is known that for finite $N$, $(1+t^2)$ times the logarithmic derivative of $E(0,(t,\infty))$ satisfies a Painlev?-VI equation, for $t\in\mathds{R}$. The orthogonal polynomials associated with the measure $w_H$ are known as well as the scaling limit of the associated kernel $K_N$, which we note $K_\infty$. One naturally expects $\lambda_1(N)$, appropriately scaled, to converge in law to the probability distribution $F_\infty(t):=\det(I-K_\infty)|_{L_2(t,\infty)}$, for $t>0$ ($t\leq0$ is not permissible in this particular case, as we will see in remark \ref{t must be bigger than 0}). We shall  see below that this is indeed the case for all values of $s$ such that $\Re(s)>-1/2$. A natural question is: does $(1+t^2)$ times the logarithmic derivative of $F_\infty(t)$ also satisfy some non-linear differential equation? And as previously mentioned, what is the rate of convergence to $F_\infty(t)$?

\subsection*{Statement of the main results}
We now state our main theorems. Our results are based on earlier work by Borodin and Olshanski \cite{Borodin-Olshanski} who obtained an explicit form for the orthogonal polynomials associated with the weight $w_H$ as well as the scaling limit for the associated kernel, and Forrester and Witte \cite{Forrester-Witte2} who express, for fixed $N$ and for any complex number $s$, with $\Re(s)>-1/2$, the probability distribution of the largest eigenvalue $\lambda_1(N)$ in terms of some non-linear differential equation. For clarity and to fix the notations, we first state a Theorem of Borodin and Olshanski \cite{Borodin-Olshanski}. We refer the reader to the paper \cite{Borodin-Olshanski} for more information on the  determinantal aspects.  The discussion on the methods we use is postponed to the end of this Section.

Borodin and Olshanski \cite{Borodin-Olshanski} give the correlation kernel for the determinantal point process defined by the measure (\ref{Hua-Pickrell on H}). In fact, the monic orthogonal polynomial ensemble $\{p_m;\;m<\Re{s}+N-\frac{1}{2}\}$ on $\mathds{R}$ associated with the weight $w_H(x)$, is defined by $p_0\equiv1$, and
\begin{equation}\label{BO polynomials (monic)}
p_m(x)=(x-i)^m\hg\left[-m,\;s+N-m,\;2\Re{s}+2N-2m;\;\frac{2}{1+ix}\right],
\end{equation}
where $\hg[a,\;b,\;c;\;x]=\sum_{n\geq0}\frac{(a)_n(b)_n}{(c)_n n!}x^n$ is the Gauss hypergeometric function, and $(x)_n=x(x+1)\ldots(x+n-1)$. Using the Christoffel-Darboux formula and the theory of orthogonal polynomials, the following was proven by Borodin and Olshanski \cite{Borodin-Olshanski}:

\begin{theorem}\label{correlation Kernel}
The $n$-point correlation function ($n\leq N$) for the eigenvalue
distribution (\ref{Hua-Pickrell on H}) is given by
$$\rho_n^{s,N}(x_1,...,x_n)=\det\left(K_{s,N}(x_i,x_j)\right)_{i,j=1}^n,$$
with the kernel $K_{s,N}(x,y)$ defined on $\mathds{R}^2$ is given
by:
\begin{equation}\label{formula corr kernel}
K_N(x,y):=K_{s,N}(x,y)=\sum_{m=0}^{N-1}\frac{p_m(x)p_m(y)}{\|p_m\|^2}\sqrt{w_H(x)w_H(y)}=\frac{\phi(x)\psi(y)-\phi(y)\psi(x)}{x-y},
\end{equation}
with
\begin{equation}\label{phi}
\phi(x)=\sqrt{Cw_H(x)}p_N(x),
\end{equation}
and
\begin{equation}\label{psi}
\psi(x)=\sqrt{Cw_H(x)}p_{N-1}(x),
\end{equation}
where
$w_H(x)=(1+ix)^{-s-N}(1-ix)^{-\overline{s}-N}=(1+x^2)^{-\Re{s}-N}e^{2\Im{s}\textrm{Arg}(1+ix)}$
and
\begin{equation}\label{BO const}
C:=C_{N,s}=\frac{2^{2\Re{s}}}{\pi}\Gamma\left[\begin{array}{ccc}
                                    2\Re{s}+N+1, & s+1, & \overline{s}+1 \\
                                    N, & 2\Re{s}+1, & 2\Re{s}+2
                                  \end{array}
\right].
\end{equation}
Here, we use the notation:
\begin{equation}
\Gamma\left[\begin{array}{cccc} a,&b,&c,&...\\
d,&e,&f,&...\end{array}\right]=\frac{\Gamma(a)\Gamma(b)\Gamma(c)\cdots}{\Gamma(d)\Gamma(e)\Gamma(f)\cdots}.
\end{equation}
Moreover, if $x=y$, the kernel is given by:
\begin{equation}\label{formula corr kernel xx}
K_N(x,x)=\phi'(x)\psi(x)-\phi(x)\psi'(x),
\end{equation}
using the Bernoulli-H?pital rule.

\end{theorem}

Note that $p_N$ is well-defined (and in $L_2(w_H)$) only for $\Re{s} > 1/2$. However, it can be analytically continued to $\Re{s} > -1/2$ using the hypergeometric expression $p_{N}(x)=(x-i)^{N}\hg[-N,\;s,\;2\Re{s};\;2/(1+ix)]$, except if $\Re{s}=0$. Moreover, Borodin and Olshanski \cite{Borodin-Olshanski} give a way to get rid of the singularity at $\Re{s}=0$. They introduce the polynomial
\begin{align}\label{p_N tilde}
\tilde{p}_N(x)&=p_N(x)-\frac{iNs}{\Re{s}(2\Re{s}+1)}p_{N-1}(x)\nonumber\\
&=(x-i)^N\hg\left[-N,s,2\Re{s}+1;\frac{2}{1+ix}\right].
\end{align}
This polynomial makes sense for any $s\in\mathds{C}$ with $\Re{s}>-1/2$ and one can define the kernel in Theorem \ref{correlation Kernel} equivalently by:
\begin{equation}\label{kernel with p tilde}
K_N(x,y)=C\frac{\tilde{p}_N(x)p_{N-1}(y)-p_{N-1}(x)\tilde{p}_N(y)}{x-y}\sqrt{w_H(x)w_H(y)}.
\end{equation}

We are interested in the distribution of the largest eigenvalue $\lambda_1(N)$ of a matrix in the $GCyE$. We have already seen that the probability that $\lambda_1(N)$ is smaller than $t$, is
\begin{equation}\label{no ev in 0-infty}
E(0,(t,\infty))=\det(I-K_N)|_{L_2(t,\infty)},
\end{equation}
for any $t\in\mathds{R}$. Hence, we need to consider the operator $K_N$ with kernel $K_N(x,y)$ restricted to the interval $(t,\infty)$ to calculate the probability that no eigenvalue is in the interval $(t,\infty)$.  This restriction is symmetric, with eigenvalues between 0 and 1. It is easy to see that $K_N$, restricted to any subinterval $J$ (or finite union of subintervals) of $\mathds{R}$, has no eigenvalue equal to 1, since $E(0,(t,\infty))>0$ for any $t\in\mathds{R}$. This is true because
$$P (\lambda_1(N) \leq t) = \textrm{cst} \int_{(- \infty,t)^N} \prod (x_j-x_k)^2 \prod w_H (x_j) dx_1\ldots dx_N,$$and the integrand is strictly positive.  Moreover, restricting the correlation function $\rho_n^{s,N}$ of Theorem \ref{correlation Kernel} to $J$ gives
\begin{align}\label{restricted correlation fct}
&\rho_n^{s,N}(x_1,...,x_n)|_{J}=\prod_{j=1}^n\chi_{J}(x_j)\rho_n^{s,N}(x_1,...,x_n)\\
=&\prod_{j=1}^n\chi_{J}(x_j)\det(K_N(x_i,x_j))_{i,j=1}^n=\det(\chi_{J}(x_i)K_N(x_i,x_j)\chi_{J}(x_j))_{i,j=1}^n,\nonumber
\end{align}
where $\chi_J$ denotes the indicator function of the set $J$. Therefore, the restriction of $K_N$ to $J$, denoted by $K_{N,J}$, defines a determinantal process on $J$ with kernel $\chi_{J}(x)K_N(x,y)\chi_J(y)=:K_{N,J}(x,y)$.

Borodin and Olshanski \cite{Borodin-Olshanski} give a scaling limit for the kernel $K_N(x,y)$ given in Theorem \ref{correlation Kernel}. Namely, $\lim_{N\rightarrow\infty}NK_N(Nx,Ny)=K_{\infty}(x,y)$, for any $x,y\in\mathds{R}^\ast=\mathds{R}\backslash \{0\}$, where the kernel $K_\infty$ is defined by
\begin{equation}\label{definition of K infinity}
K_{\infty} (x,y) = \frac{1}{2\pi} \frac{\Gamma(s+1) \Gamma(\overline{s} + 1) }
{\Gamma(2 \Re{s} + 1) \Gamma(2 \Re{s} + 2)} \frac{\tilde{P} (x) Q(y) - Q(x) \tilde{P} (y) }{x-y},
\end{equation}
if $x\neq y$, and,
\begin{equation}\label{definition of K infinity for x=x}
K_{\infty}(x,x)=\frac{1}{2\pi} \frac{\Gamma(s+1) \Gamma(\overline{s} + 1) }
{\Gamma(2 \Re{s} + 1) \Gamma(2 \Re{s} + 2)}(\tilde{P}'(x)Q(x)-Q'(x)\tilde{P}(x)),
\end{equation}
where
\begin{align*}
\tilde{P} (x) &= |2/x|^{\Re{s} } e^{-i/x + \pi \Im{s} \textrm{Sgn} (x) /2}  {}_1F_1 \left[s, 2 \Re{s} + 1;
\frac{2i}{x} \right],\\
Q (x) &= (2/x)|2/x|^{\Re{s} } e^{-i/x + \pi \Im{s} \textrm{Sgn} (x) /2} {}_1F_1 \left[s+1, 2 \Re{s} + 2;
\frac{2i}{x} \right],
\end{align*}
with
\begin{equation*}
{}_1F_1 \left[r, q; x \right]=\sum_{n\geq0}\frac{(r)_n}{(q)_nn!}x^n,
\end{equation*}
for any $r,q,x\in\mathds{C}$.

\begin{remark}\label{K_infty defines a det process}
The kernel $K_\infty$ defines a determinantal point process (see \cite{Borodin-Olshanski} , Theorems IV and 6.1).
\end{remark}

\begin{remark}\label{sine kernel}
If $s=0$, the limiting kernel $K_\infty$ writes as
\begin{equation*}
K_\infty(x_1,x_2)=\frac{1}{\pi}\frac{\sin(1/x_2-1/x_1)}{x_1-x_2}.
\end{equation*}
Under the change of variable $y=\frac{1}{\pi x}$ and taking into account the corresponding change of the differential $dx$, $K_\infty$ translates to the famous sine kernel with correlation function
\begin{equation*}
\rho_n(y_1,\ldots,y_n)=\det\left(\frac{\sin(\pi(y_i-y_j))}{\pi(y_i-y_j)}\right)_{i,j=1}^n,
\end{equation*}
for any $n\in\mathds{N}$ and $y_1,\ldots,y_n\in\mathds{R}$ (see Borodin and Olshanski \cite{Borodin-Olshanski}). \end{remark}
Before stating our main results, we need to introduce one more notation: we note $K_{[N]} (x,y)$ the kernel
\begin{equation}\label{KNN}
K_{[N]} (x,y) := N K_N (Nx,Ny),
\end{equation}and $K_{[N]}$ the associated integral operator. We also recall the definition of the Fredholm determinant: if $K$ is an integral operator with kernel given by $K(x,y)$, then the $k$-correlation function $\rho_k$  is defined by:
\begin{equation*}
\rho_k(x_1,\ldots,x_k) = \det (K(x_i,x_j)_{1 \leq i,j \leq k}).
\end{equation*}
The Fredholm determinant $F$,  from $\mathds{R}_+^*$ to $\mathds{R}$, is then defined by
\begin{equation}\label{Rubin}
F(t) = 1 + \sum_{k \geq 1} \frac{(-1)^k}{k!} \int_{(t, \infty)^k} \rho_k(x_1,\ldots,x_k) dx_1\ldots dx_k.
\end{equation}

\begin{theorem} \label{Fredholm2}
For $s$ such that $\Re{s}>-1/2$ and $t>0$, let $F_N$ be the Fredholm determinant associated with $K_{[N]}$, and let $F_{\infty}$ be the Fredholm determinant associated with $K_\infty$. Then, $F_N$ and $F_{\infty}$ are in $\mathcal{C}^3(\mathds{R}_+^*, \mathds{R})$, and for $p \in \{0,1,2,3\}$, the $p$-th derivative of
$F_N$ (with respect to $t$) converges pointwise to the $p$-th derivative of $F_{\infty}$.
\end{theorem}
As an immediate consequence, one obtains the following convergence in law for the re-scaled largest eigenvalue:

\begin{corollary}\label{scaling limit for largest eigenvalue with fixed parameter s}
Given the set of $N\times N$ random Hermitian matrices $H(N)$ with the generalized Cauchy probability distribution (\ref{Full Hua-Pickrell}), denote by $\lambda_1(N)$ the largest eigenvalue of such a randomly chosen matrix. Then, the law of $\lambda_1(N)/N$ converges to the distribution of the largest point of the determinantal process on $\mathds{R}^\ast$ described by the limiting kernel $K_\infty(x,y)$ in the following sense:
\begin{equation*}
P\left[\frac{\lambda_1(N)}{N}\leq x_0\right]=\det(I-K_N)|_{L_2(Nx_0,\infty)}\longrightarrow\det(I-K_\infty)|_{L_2(x_0,\infty)},\quad\textrm{as }N\rightarrow\infty,
\end{equation*}
for any $x_0>0$.
\end{corollary}
\begin{remark}\label{t must be bigger than 0}
Note that in the case of finite $N$, the range of the largest eigenvalue is the whole real line, whereas in the limit case when $N\rightarrow\infty$, the range of the largest eigenvalue is $\mathds{R}_+^*$. This is because an infinite number of points accumulate close to 0 (0 itself being excluded however). The accumulation of the points can be seen from the fact that due to the form of $K_\infty(x,x)$ (see (\ref{definition of K infinity for x=x})), $\lim_{\epsilon\rightarrow\infty}\int_\epsilon^\infty K_\infty(x,x)dx$ diverges.
\end{remark}

Now, define
\begin{equation}\label{theta_infty}
\theta_{\infty}(\tau)=\tau\frac{d\log\det(I-K_{\infty})|_{L_2(\tau^{-1},\infty)}}{d\tau},\qquad \tau>0.
\end{equation}
Using the result of Forrester and Witte \cite{Forrester-Witte2} for the distribution of the largest eigenvalue for fixed $N$ and Theorem \ref{Fredholm2}, we are able to show:
\begin{theorem}\label{convergence of the N-solution to the infinity-solution}
Let $s$ be such that $\Re{s} > -1/2$. Then the function $\theta_{\infty}$ given by (\ref{theta_infty}) is well defined and is a solution to the Painlev?-V equation  on $\mathds{R}_+^*$:
\begin{align}\label{Painleve V for theta}
-\tau^2(\theta''(\tau))^2=&\left[2(\tau\theta'(\tau)-\theta(\tau))+(\theta'(\tau))^2+i(\overline{s}-s)\theta'(\tau)\right]^2\nonumber\\
&-(\theta'(\tau))^2(\theta'(\tau)-2is)(\theta'(\tau)+2i\overline{s}).
\end{align}

\end{theorem}

\begin{remark}
This implies in particular the result of Jimbo, Miwa, M?ri and Sato \cite{JMMS} that the sine kernel, which is the special case of the $K_\infty$ kernel with parameter $s=0$ (see remark \ref{sine kernel}), satisfies the Painlev?-V equation \eqref{Painleve V for theta} with $s=0$.
\end{remark}
Eventually, following our initial motivation, we have the following result about the rate of convergence:

\begin{theorem}\label{rate of convergence for large x}
For all $x_0 > 0$, and for $x>x_0$,
\begin{equation*}
\left|P\left[\frac{\lambda_1(N)}{N}\leq x\right]-\det(I-K_\infty)|_{L_2(x,\infty)}\right|\leq\frac{1}{N}C(x_0,s),
\end{equation*}where $C(x_0,s)$ is a constant depending only on $x_0$ and $s$.
\end{theorem}

Now, we say a few words about the way we prove the above theorems. Our proofs, splitted into several technical lemmas, only use elementary methods; namely, they only involve checking pointwise convergence and domination in all the quantities involved in the Fredholm determinants of $K_{[N]}$ and $K_\infty$. We can then apply dominated convergence to show that the logarithmic derivative of the Fredholm determinant of $K_{[N]}$, as well as its derivatives, converge pointwise to the respective derivatives of the Fredholm determinant of $K_\infty$. This suffices to show that the Fredholm determinant of $K_\infty$ satisfies a Painlev?-V equation because we can write the rescaled finite $N$ Painlev?-VI equation of Forrester and Witte (\cite{Forrester-Witte2}, \cite{Forrester-Witte3} and Theorem \ref{Rubin2} below) as the sum of polynomial functions of the Fredholm determinant of $K_{[N]}$ and their first, second and third derivatives. Moreover, the various estimates and bounds we obtain for the different determinants and functions involved in our problem help us to obtain directly an estimate for the rate of convergence in Corollary \ref{scaling limit for largest eigenvalue with fixed parameter s} (that is Theorem \ref{rate of convergence for large x}).

Given Theorem \ref{correlation Kernel} and the Painlev\'e VI characterization of Forrester and Witte \cite{Forrester-Witte2}, the results contained  in Theorem \ref{Fredholm2} and Corollary \ref{scaling limit for largest eigenvalue with fixed parameter s} are very natural; but yet they have to be rigourously checked. As far as Theorem \ref{convergence of the N-solution to the infinity-solution} is concerned,  Borodin and Deift \cite{Borodin-Deift} obtain  the same equation as (\ref{Painleve V for theta}) from the scaling limit of a Painlev?-VI equation characterizing a general 2F1-kernel similar to our kernel $K_N$ (Section 8 in  \cite{Borodin-Deift}). They claim that it is natural to expect that the appropriately scaled logarithmic derivative of the Fredholm determinant of their 2F1-kernel solves this Painlev?-V equation. In fact, according to our Theorem \ref{convergence of the N-solution to the infinity-solution}, (\ref{theta_infty}) corresponds to their limit, when $N\to\infty$, of the scaled solution of the Painlev?-VI equation  and solves the Painlev\'e-V equation (\ref{Painleve V for theta}). Borodin and Deift's method is based on the combination of Riemann-Hilbert theory with the method of isomonodromic deformation of certain linear differential equations. The method is very powerful and general; however, we were not able to apply it in our situation; moreover, it seems that we would have to restrict ourselves to the values of $s$ such that $0\leq\Re(s)\leq1$. As we shall mention it later in the paper, all the ingredients seem to be there to apply the method of Tracy and Widom \cite{Tracy-Widom}; here again, we were not able to find our way: it seems to us (see next Section) that with this method, we could obtain at best a second order non-linear differential equation for $\theta_\infty$, which is equivalent to (\ref{Painleve V for theta}), but for a restricted range of $s$: $\Re(s)>1/2$, and thus excluding the case $s=0$ of  the sine kernel. On the other hand, our method to prove Theorem \ref{convergence of the N-solution to the infinity-solution} heavily  relies on the result of Forrester and Witte  \cite{Forrester-Witte2} for fixed $N$: hence we do not provide a general method to obtain Painlev\'e equations. However, it is an efficient approach to obtain some information about the rate of convergence in Corollary \ref{scaling limit for largest eigenvalue with fixed parameter s}.

\section{Proof of Theorems \ref{Fredholm2}, and \ref{convergence of the N-solution to the infinity-solution}}
In this Section, we split the proofs of Theorems \ref{Fredholm2}, and \ref{convergence of the N-solution to the infinity-solution} into several technical Lemmas. The notations are those introduced in Section \ref{intro}. Throughout this paper, the notation $C(a_0,a_1,\ldots,a_n)$ stands for a positive constant which only depends on the parameters $a_0,a_1,\ldots,a_n$, and whose value may change from line to line (we shall not be interested in explicit values for the different constants). We first bring in an ODE that $\theta_\infty$ should satisfy; then we prove several technical lemmas about the convergence of the correlation functions and the derivatives of the kernel $K_{[N]}$. We shall use these lemmas to show that $\theta_\infty(t)$ is indeed well defined (i.e. $F_\infty(t)$ is non-zero for any $t>0$) and
to prove Theorems  \ref{Fredholm2} and \ref{convergence of the N-solution to the infinity-solution}.
\subsection{Scaling Limits}\label{Scaling Limits}
We now state a result by Forrester and Witte \cite{Forrester-Witte2} which will play an important role in the proof of Theorem \ref{convergence of the N-solution to the infinity-solution}.
\begin{theorem}\label{Rubin2}
For $\Re(s)>-1/2$, define
\begin{align*}
\sigma(t)=&(1+t^2)\frac{d}{dt}\log\det(I-K_N)|_{L_2(t,\infty)}\\
=&(1+t^2)\frac{d}{dt}\log P[\textrm{there is no eigenvalue inside
}(t,\infty)].
\end{align*}
Then, for $t\in\mathds{R}$, $\sigma(t)$ satisfies the equation:
\begin{align}\label{ODE for sigma 2}
(1+t^2)(\sigma'')^2+4(1+t^2)(\sigma')^3-8t(\sigma')^2\sigma+4\sigma^2(\sigma'-(\Re{s})^2)+8 (t(\Re{s}^2) - \Re{s} \Im{s}\nonumber \\
- N \Im{s} )\sigma \sigma'+4 ( 2t \Im{s} (N+ \Re{s}) - (\Im{s})^2 - t^2 (\Re{s})^2 + N (2 \Re{s} + N) ) (\sigma')^2 = 0.
\end{align}
\end{theorem}
\begin{remark}
The ODE (\ref{ODE for sigma 2}) is equivalent to the master Painlev?
equation (SD-I) of Cosgrove and Scoufis \cite{Cosgrove-Scoufis}.
Cosgrove and Scoufis, show that the solution of this equation can be
expressed in terms of the solution of a Painlev?-VI equation using a
B?cklund transform. In the case of $s$ real, this transformation is described in Forrester and Witte \cite{Forrester-Witte}.
\end{remark}
\begin{remark}
One can also attempt to use the general method introduced by Tracy
and Widom in \cite{Tracy-Widom} for kernels of the form (\ref{formula corr kernel})  to prove the above Theorem. Their method establishes a system of PDE's, the so called Jimbo-Miwa-M?ri-Sato equations, which can be reduced to a Painlev?-type equation. The PDE's consist of a set of universal equations and a set of equations depending on the specific form of the following recurrence differential equation for $\phi$ and $\psi$:
\begin{equation}\label{General recurrence equation for phi and psi}
\begin{array}{c}
 m(x)\phi'(x)=A(x)\phi(x)+B(x)\psi(x) \\[5pt]
 m(x)\psi'(x)=-C(x)\phi(x)-A(x)\psi(x),
\end{array}
\end{equation}
where $A,B,C$ and $m$ are polynomials in $x$. Doing the calculations for the case of the determinantal process with kernel $K_N$, one obtains that for $\phi$ and $\psi$ given in Theorem \ref{correlation Kernel}, the recurrence equations (\ref{General recurrence equation for phi and
psi}) hold with:
\begin{align*}
m(x)&=1+x^2,\\
A(x)&=-x\Re{s}+\Im{s}\left(1+\frac{N}{\Re{s}}\right),\\
B(x)&=\frac{|s|^2}{\Re{s}^2}N\frac{2\Re{s}+N}{2\Re{s}+1},\\
C(x)&=2\Re{s}+1.
\end{align*}
Note that this only makes sense if $\Re{s}\neq0$.
One can then show that for $t\in\mathds{R}$, and $\Re(s)>1/2$, the equation (\ref{ODE for sigma 2}) holds. In the case of $s\in(1/2,\infty)\subset\mathds{R}$, this Theorem was obtained in this way by Forrester and Witte in \cite{Forrester-Witte} (Proposition 4). We would like to shortly explain the reason why we were able to make this method work only for $\Re(s)>1/2$. Indeed, the method of Tracy and Widom  has originally been developed for finite intervals (or unions of finite intervals). If one applies the method to the case of a semi-infinite interval $(t,\infty)$, one has to consider an interval $(t,a)$, where $a>t$. Then, one writes down the PDE's of Tracy and Widom for that interval and takes the limit in all the equations as $a\rightarrow\infty$. Note that the variables in these PDE's are the end-points $t$ and $a$ of the interval. It is clear, that one has to be careful about the convergence of the quantities involved in these equations, when $a\rightarrow\infty$. In particular, one needs in our case that  the term $(1+a^2)Q(a)R(t,a)$, where $R(x,y)$ is the  kernel of the resolvent operator $K_{N,J}(1-K_{N,J})^{-1}$, and $Q(x)=(I-K_{N,J})^{-1}\phi(x)$, which is of order $a^{1-2\Re{s}}$, tends to zero, when $a\rightarrow\infty$. This implies the restriction $\Re{s}>1/2$. One might encounter the same type of obstacle in an attempt to prove Theorem \ref{convergence of the N-solution to the infinity-solution} with this method (we will give  the corresponding recurrence equations for $\phi$ and $\psi$ in the case of $K_\infty$ in Remark \ref{infinity solution via Tracy-Widom}).
\end{remark}

We now show  that when $N\rightarrow\infty$, the ODE (\ref{ODE for sigma 2}) converges to a
$\sigma$-version of the Painlev?-V equation. This limiting equation is also given in Borodin and Deift \cite{Borodin-Deift} (Proposition 8.14). Borodin and Deift obtain this equation as a scaling limit of a Painlev?-VI equation characterizing their 2F1-kernel. However, their 2F1-kernel is different from our kernel $K_N$.

Set for $\tau>0$,
\begin{equation}\label{theta (change of variable for the ODE)}
\theta(\tau):=\theta_N(\tau):=\tau\frac{d\log\det(1-K_N)|_{L_2(N\tau^{-1},\infty)}}{d\tau},
\end{equation} where  $R(x,y)$ is the  kernel of the resolvent operator $K_{N,J}(1-K_{N,J})^{-1}$.
Then,
\begin{equation*}
\theta(\tau)=\tau\left(-\frac{N}{\tau^2}\right)R\left(\frac{N}{\tau},\frac{N}{\tau}\right)=-\frac{N}{\tau}\left[\frac{\sigma\left(\frac{N}{\tau}\right)}{1+\frac{N^2}{\tau^2}}\right].
\end{equation*}
It follows that
\begin{align*}
\sigma\left(\frac{N}{\tau}\right)=&-\theta(\tau)\left(\frac{\tau}{N}+\frac{N}{\tau}\right),\\
\sigma'\left(\frac{N}{\tau}\right)=&\frac{\tau^2}{N^2}(\tau\theta'(\tau)+\theta(\tau))+(\tau\theta'(\tau)-\theta(\tau)),\\
\sigma''\left(\frac{N}{\tau}\right)=&-\frac{\tau^3}{N^3}[4\tau\theta'(\tau)+2\theta(\tau)+\tau^2\theta''(\tau)]-\frac{\tau^3}{N}\theta''(\tau).
\end{align*}
Now, put this into the ODE (\ref{ODE for sigma 2}) with
$t=\frac{N}{\tau}$. After dividing by $N^2$, we obtain:
\begin{align*}
&\left(\frac{1}{\tau^2}\right)^2(\tau^3\theta''(\tau))^2+4\left(\frac{1}{\tau^2}\right)(\tau\theta'(\tau)-\theta(\tau))^3+\frac{8}{\tau}(\tau\theta'(\tau)-\theta(\tau))^2\frac{\theta(\tau)}{\tau}\\
&+4\left(\frac{\theta(\tau)}{\tau}\right)^2(\tau\theta'(\tau)-\theta(\tau)-(\Re{s})^2)-8\left(\frac{(\Re{s})^2}{\tau}-\Im{s}\right)\frac{\theta(\tau)}{\tau}(\tau\theta'(\tau)-\theta(\tau))\\
&+4\left[2\frac{\Im{s}}{\tau}-\frac{(\Re{s})^2}{\tau^2}+1\right](\tau\theta'(\tau)-\theta(\tau))^2= O(N^{-1}).
\end{align*}
This gives
\begin{align*}
-\tau^2&(\theta''(\tau))^2=4\left\{(\theta'(\tau))^2(\tau\theta'(\tau)-\theta(\tau)-(\Re{s})^2)+2\Im{s}\,\theta'(\tau)(\tau\theta'(\tau)-\theta(\tau))\right.\\
&\left.+(\tau\theta'(\tau)-\theta(\tau))^2\right\}+O(N^{-1}).
\end{align*}
Now if one neglects the terms of order $O(N^{-1})$, it is easy to see that this is precisely equation (\ref{Painleve V for theta}).  But this is also exactly the $\sigma$-form of the Painlev?-V equation in Borodin and Deift \cite{Borodin-Deift}, Proposition 8.14.

Hence, $\theta_N(\tau)(=\theta(\tau))$ satisfies a differential equation which
tends to the $\sigma$-Painlev?-V equation and we have the following Proposition:

\begin{proposition}\label{convergence of the PVI to the PV equation}
The ODE (\ref{ODE for sigma 2}) with the change of variable $t=N/\tau$, $\tau>0$, is solved by $\theta_N(\tau)$, and is of the form
\begin{equation*}
\sum_{k=0}^m N^{-k}\frac{P_k(\tau,\theta_N(\tau),\theta_N'(\tau),\theta_N''(\tau))}{\tau^q}=0,
\end{equation*}
where $m$ and $q$ are universal integers and the $P_k$'s are polynomials which are independent of $N$. Moreover, $P_0(\tau,\theta_N(\tau),\theta_N'(\tau),\theta_N''(\tau))\tau^{-q}$ corresponds to the $\sigma$-form of the Painlev?-V
equation (\ref{Painleve V for theta}).
\end{proposition}
\begin{remark}
We note that $\theta_N(\tau)$, given by (\ref{theta (change of variable for the ODE)}), is a solution of the ODE (\ref{ODE for sigma 2}), with $t=N/\tau$. Moreover, we know that $\lim_{N\rightarrow\infty}NK_N(x,y)=K_\infty(x,y)$, for any $x,y\in\mathds{R}^\ast$. Hence it is natural to guess that $\theta_\infty(\tau)$ should satisfy the ODE (\ref{Painleve V for theta}).
\end{remark}

\subsection{Some technical Lemmas}

For clarity, we decompose the proof of our Theorems into several Lemmas about the convergence of correlation functions and the derivatives of the kernel $K_{[N]}$.

\begin{lemma}\label{key estimates for convergence theorem}
Let $K$ be a function in $C^2 ((\mathds{R}_+^*)^2, \mathds{R})$, such that for all $k \in \mathds{N}$, and $x_1,x_2,\ldots,x_k > 0$, the matrix $K(x_i,x_j)_{1 \leq i,j \leq k}$ is symmetric and positive. Define the $k$-correlation function $\rho_k$ by:
\begin{equation*}
\rho_k(x_1,\ldots,x_k) = \det (K(x_i,x_j)_{1 \leq i,j \leq k}),
\end{equation*}
and suppose that for $(p,q) \in \{(i,j);\;i,j\in\mathds{N}_0,\; i+j\leq2\}$, for some $\alpha>1/2$, and for all $x_0>0$, one has
the upper bound
\begin{equation}\label{bound on K for p+q}
\left|\frac{\partial^{p+q}}{\partial x^p \partial y^q} K(x,y) \right| \leq \frac{C(x_0)}{(xy)^{\alpha}},
\end{equation}
if $x,y \geq x_0$. Then, $\rho_k$ is in $C^2 ((\mathds{R}_+^*)^k, \mathds{R})$ for all $k$, and for all $x_0 > 0$, $x_1,\ldots,x_k \geq x_0$,
one has:
\begin{equation}
\left|\frac{\partial^{p}}{\partial x_j^p} \rho_k (x_1,\ldots,x_k)\right| \leq \frac{(C(x_0))^k}{(x_1\ldots x_k)^{2 \alpha}}, \label{main bound for K}
\end{equation}
if $p \in \{0,1,2\}$ and $j \in \{1,\ldots,k\}$. Moreover,
\begin{equation}
\frac{\partial^{p}}{\partial x_j^p} \rho_k (x_1,\ldots,x_k) = 0 \label{partial rho equal 0}
\end{equation}
if $p \in \{0,1\}$, $j \in \{1,\ldots,k\}$ and if there exists $j' \neq j$ such that $x_j = x_{j'}$.
\end{lemma}

\begin{proof}
Fix $k\in\mathds{N}$. The fact that $\rho_k$ is in $C^2$ is an immediate consequence of the fact that $K$ is in $C^2$.
For $x_1,...,x_{j-1},x_{j+1},...,x_k$ fixed, the function:
\begin{equation*}
t \mapsto \rho_k(x_1,...,x_{j-1},t,x_{j+1},...,x_k)
\end{equation*}
is positive by the positivity of $K$, and equal to zero if $t=x_{j'}$ for some $j' \in\{1,\ldots,j-1,j+1,\ldots,k\}$. Therefore,
$t=x_{j'}$ is a local minimum of this function and one deduces the equality \eqref{partial rho equal 0}. We now turn to the proof of \eqref{main bound for K}. By symmetry of $\rho_k$, we only need to show the case $j=1$. We isolate the terms containing $x_1$ in the determinant defining $\rho_k$ to obtain:
\begin{align*}
& \rho_k (x_1,\ldots,x_k) = K(x_1,x_1) \det( K(x_{l+1},x_{m+1})_{1 \leq l,m \leq k-1})\\
&+ \sum_{2 \leq i,j \leq k} (-1)^{i+j-1} K(x_i,x_1) K(x_1,x_j)
\det( K(x_{l+1+ \mathds{1}_{l \geq i-1}},x_{m+1 + \mathds{1}_{m \geq j-1} })_{1 \leq l,m \leq k-2}),
\end{align*}
where we take the convention that an empty sum is equal to $0$ and an empty determinant is equal to $1$.
One deduces:
\begin{align*}
& \frac{\partial}{\partial x_1}\rho_k (x_1,\ldots,x_k) = (K'_1 + K'_2)(x_1,x_1)
\det( K(x_{l+1},x_{m+1})_{1 \leq l,m \leq k-1}) \\
& + \sum_{2 \leq i,j \leq k} (-1)^{i+j-1} (K'_2(x_i,x_1) K(x_1,x_j) + K(x_i,x_1) K'_1(x_1,x_j)) \\
& \det( K(x_{l+1+ \mathds{1}_{l \geq i-1}},x_{m+1 + \mathds{1}_{m \geq j-1} })_{1 \leq l,m \leq k-2}),
\end{align*}
and
\begin{align*}
& \frac{\partial^2}{\partial x_1^2}\rho_k (x_1,...,x_k) = (K''_{1,1} + 2 K''_{1,2} + K''_{2,2})(x_1,x_1)
\det( K(x_{l+1},x_{m+1})_{1 \leq l,m \leq k-1})  \\
&+ \sum_{2 \leq i,j \leq k} (-1)^{i+j-1} (K''_2(x_i,x_1) K(x_1,x_j) + 2 K'_2(x_i,x_1) K'_1(x_1,x_j)\\
&+ K(x_i,x_1) K''_1(x_1,x_j)) \det( K(x_{l+1+ \mathds{1}_{l \geq i-1}},x_{m+1 + \mathds{1}_{m \geq j-1} })_{1 \leq l,m \leq k-2}),
\end{align*}
where for $p,q \in \{1,2\}$, $K'_p$ denotes the derivative of $K$ with respect to the $p$-th variable,
and $K''_{p,q}$ denotes the second derivative of $K$ with respect to the $p$-th and the $q$-th variable.
By the positivity of $K$, there exists, for all $r \in \mathds{N}$ and $y_1,...,y_r, z_1,...,z_r>0$, vectors
$e_1,...,e_r,f_1,...,f_r$ of an Euclidian space $E$ equipped with its usual scalar product $(.|.)$, such that $(e_i|f_j) = K(y_i,z_j)$ for all $i,j\in\{1,\ldots,r\}$.
Now, we can define a scalar product on the $r$-th exterior power  of $E$ by setting
\begin{equation*}
(u_1 \wedge ... \wedge u_r| v_1 \wedge ... \wedge v_r)  = \det ( (u_i|v_j)_{1 \leq i,j \leq r}),
\end{equation*}
for all $u_1,...,u_r,v_1,...,v_r \in E$.
%Using the Cauchy-Schwarz inequality for this scalar product, one deduces:
%\begin{equation*}
%|\det ( (e_i|f_j)_{1 \leq i,j, \leq r}) | \leq \sqrt{  \det ( (e_i|e_j)_{1 \leq i,j, \leq r})
%\det ( (f_i|f_j)_{1 \leq i,j, \leq r})}.
%\end{equation*}
Note that this scalar product is nothing else than a Gram determinant and we have the upper bound
\begin{equation*}
|\det ( (e_i|f_j)_{1 \leq i,j \leq r}) | \leq \prod_{i=1}^r \|e_i\|_E \prod_{i=1}^r \|f_i\|_E,
\end{equation*}
$\|.\|_E$ being the norm associated to $(.|.)$. This last bound is equivalent to
\begin{equation}
|\det ( K(y_i,z_j)_{1 \leq i,j \leq r}) | \leq  \sqrt{ \prod_{i=1}^r K(y_i,y_i) \prod_{i=1}^r
K(z_i,z_i)}.  \label{bound on the determinant (in proof of key lemma)}
\end{equation}
Now, let $x_0> 0$ and $x_1,\ldots,x_k \geq x_0$. The bound \eqref{bound on K for p+q} given in the statement of the Lemma and the inequality \eqref{bound on the determinant (in proof of key lemma)} imply
\begin{equation*}
|\det( K(x_{l+1},x_{m+1})_{1 \leq l,m \leq k-1})| \leq \frac{(C(x_0))^{k-1}}{(x_2 x_3\cdots x_k)^{2 \alpha}}
\end{equation*}
and
\begin{align*}
& |\det( K(x_{l+1+ \mathds{1}_{l \geq i-1}},x_{m+1 + \mathds{1}_{m \geq j-1} })_{1 \leq l,m \leq k-2})|\\
\leq& \frac{ (C(x_0))^{k-2} }{ (x_2 x_3\cdots x_{i-1}x_{i+1}\cdots x_k)^{\alpha}
(x_2 x_3\cdots x_{j-1}x_{j+1}\cdots x_k)^{\alpha} } \\
=& \frac{( C(x_0))^{k-2} (x_i x_j)^{\alpha} }{ (x_2...x_k)^{2 \alpha}}.
\end{align*}
Hence, each term involved in the expressions of $\rho_k$ and its two first derivatives with respect to $x_1$
is smaller than $4( C(x_0))^k / (x_1\cdots x_k)^{2 \alpha}$ and therefore, the absolute values of $\rho_k$ an its derivatives are bounded by $4 ((k-1)^2 + 1) ( C(x_0))^k /(x_1\cdots x_k)^{2 \alpha} \leq 4^k ( C(x_0))^k /(x_1\cdots x_k)^{2 \alpha} $, implying the bound \eqref{main bound for K}.
\end{proof}

\begin{remark} \label{C(x_0)}
In the above proof, the value of $C(x_0)$ does not change. It is thus possible to take $C(x_0)$ in the inequality \eqref{main bound for K} to be equal to $4$ times the value of $C(x_0)$ in \eqref{bound on K for p+q}.
\end{remark}

We now have to prove that the re-scaled kernel $K_{[N]}$  satisfies the hypothesis of Lemma \ref{key estimates for convergence theorem}, and that its partial derivatives converge pointwise to the partial derivatives of $K_\infty$. In the following, we introduce the notation
\begin{equation*}
F_{n,h,a}(x)={}_2F_1\left[-n,h,a;2/(1+ix)\right],
\end{equation*}
for $(n,h,a)\in\mathds{N}\times\mathds{C}\times\mathds{R}_+^\ast$.

\begin{lemma} \label{derivatives F_{n,h,a}}
Let $\epsilon \in \{0,1\}$, $h \in \mathds{C}$, $a \in \mathds{R}_+^*$. For $N \in \mathds{N}$, we set
$n := N - \epsilon$. Then, $x \mapsto F_{n,h,a} (Nx)$ and $x \mapsto \,_1F_1 [h,a;2i/x]$
are in $C^{\infty} (\mathds{R}^*)$, and for all $p \in \mathds{N}$ and $x \in \mathds{R}^*$:
\begin{equation*}
\frac{d^p}{dx^p} (F_{n,h,a} (Nx)) \underset{N \rightarrow \infty}{\longrightarrow}  \frac{d^p}{dx^p} (\,_1F_1 [h,a;2i/x]).
\end{equation*}
Moreover, for all $x_0 > 0$ and for all $x\in\mathds{R}$ such that $|x|\geq x_0$, one has the bound
\begin{equation*}
\left| \frac{d^p}{dx^p} (F_{n,h,a} (Nx)) \right| \leq \frac{C(x_0, h,a,p)}{|x|^{p + \mathds{1}_{p > 0}}}.
\end{equation*}
\end{lemma}

\begin{proof}
One has
\begin{equation*}
F_{n,h,a} (Nx) = \sum_{k=0}^{\infty} \frac{(-n)_k (h)_k}{(a)_k k!} \, \left( \frac{2}{1 + Nix} \right)^k,
\end{equation*}
where only a finite number of the summands are different from zero. This implies that the function
is $C^{\infty}$ on $\mathds{R}^\ast$, and
\begin{equation*}
\frac{d^p}{dx^p} (F_{n,h,a} (Nx)) = \sum_{k=0}^{\infty} \frac{(-n)_k (h)_k}{(a)_k k!}
(k)_p \left( \frac{2}{1 + Nix} \right)^{k+p} \left( - \frac{iN}{2} \right)^p.
\end{equation*}
The term of order $k$ in this sum is dominated by (note that $a>0$)
\begin{equation*}
\frac{(|h|)_k}{(a)_k k!} \, (k)_p \frac{2^k}{|x|^{k+p}},
\end{equation*}
and for fixed $x$, tends to
\begin{equation*}
\frac{(h)_k}{(a)_k k!} (k)_p \frac{(2i)^k (-1)^p} {x^{k+p}},
\end{equation*}
when $N\rightarrow\infty$.
One deduces, that for $|x| \geq x_0 > 0$:
\begin{align*}
\left| \frac{d^p}{dx^p} (F_{n,h,a} (Nx)) \right| & \leq \sum_{k=0}^{\infty} \frac{(|h|)_k}{(a)_k k!}  (k)_p \frac{2^k}{|x|^{k+p}} \\
& \leq \mathds{1}_{p = 0} + \frac{1}{|x|^{p+1}} \sum_{k=1}^{\infty}
\frac{(|h|)_k}{(a)_k k!} (k)_p \frac{2^k}{x_0^{k-1}}\\
& \leq \frac{C(x_0,h,a,p)}{|x|^{p + \mathds{1}_{p > 0}}}
\end{align*}
which is the desired bound. Now, by dominated convergence, one has
\begin{equation*}
\frac{d^p}{dx^p} (F_{n,h,a} (Nx))  \underset{N \rightarrow \infty}{\longrightarrow} \sum_{k=0}^{\infty} \frac{(h)_k}{(a)_k k!} (k)_p \frac{(2i)^k (-1)^p} {x^{k+p}}.
\end{equation*}
Hence, Lemma \ref{derivatives F_{n,h,a}} is proved if we show that $x \mapsto _1F_1 [h,a;2i/x]$ is $C^{\infty}$ on $\mathds{R}^\ast$, and that
\begin{equation}
\frac{d^p}{dx^p} (\,_1F_1 [h,a;2i/x]) = \sum_{k=0}^{\infty}  \frac{(h)_k}{(a)_k k!} (k)_p \frac{(2i)^k (-1)^p} {x^{k+p}}. \label{sum1F1}
\end{equation}
But the sum in \eqref{sum1F1} is obtained by taking the derivative of order $p$ of each term of the sum defining $_1F_1$.
Therefore, we are done, since this term by term derivation is justified by the domination of the right hand side of \eqref{sum1F1} by $C(x_0,h,a,p)/|x|^{p+\mathds{1}_{p>0}}$ on $\mathds{R} \backslash (-x_0,x_0)$.
\end{proof}

\begin{lemma} \label{derivatives P_N Q_N}
Fix $s$ such that $\Re{s}>-\frac{1}{2}$. Define the functions $\tilde{P}_N$ and $Q_N$ by
\begin{align*}
\tilde{P}_N (x) &= 2^{\Re{s}} \left( \frac{\Gamma (2 \Re{s} + N + 1)}{N \Gamma(N)} \right)^{1/2} \tilde{p}_N
(Nx) \sqrt{w_H(Nx)},\\
Q_N(x) &= 2^{\Re{s}+1} \left( \frac{N \Gamma (2 \Re{s} + N + 1)}{\Gamma(N)} \right)^{1/2} p_{N-1}
(Nx) \sqrt{w_H(Nx)},
\end{align*}
where $\tilde{p}_N$, $p_{N-1}$ and $w_H$ are given in Theorem \ref{correlation Kernel} and the remark below that Theorem. Then, $\tilde{P}_N$ and $Q_N$ are $C^{\infty}$ on $\mathds{R}$, $\tilde{P}$ and $Q$, defined below \eqref{definition of K infinity}, are $C^{\infty}$ on
$\mathds{R}^*$, and for all $x \in \mathds{R}^*$, $p \in \mathds{N}_0$,
\begin{align*}
&(\textrm{Sgn} (x))^N \tilde{P}_N^{(p)} (x) \underset{N \rightarrow \infty}
{\longrightarrow} \tilde{P}^{(p)} (x),\\
&(\textrm{Sgn} (x))^N Q_N^{(p)} (x) \underset{N \rightarrow \infty}
{\longrightarrow} Q^{(p)} (x).
\end{align*}
Moreover, for all $p \in \mathds{N}_0$, $x_0>0$, one has the following bounds:
\begin{equation*}
\left|\tilde{P}_N^{(p)} (x)\right| \leq \frac{C(x_0,s,p)}{|x|^{p + \Re{s}}},
\end{equation*}
and
\begin{equation*}
\left|Q_N^{(p)} (x)\right| \leq \frac{C(x_0,s,p)}{|x|^{p + 1 + \Re{s}}},
\end{equation*}
for all $|x| \geq x_0$.
\end{lemma}

\begin{proof}
We define
\begin{equation*}
\Phi_N(x) = D(N,n,s) (Nx-i)^n F_{n,h,a} (Nx) (1 + iNx)^{(-s-N)/2} (1-iNx)^{(-\overline{s}-N)/2},
\end{equation*}
where
\begin{equation*}D(N,n,s) = 2^{\Re{s} + (N-n)} \left( \frac{\Gamma(2 \Re{s} + N + 1)}{N \Gamma(N)} \right)^{1/2} N^{N-n},
\end{equation*}
and $N-n\in\{0,1\}$ (see Lemma \ref{derivatives F_{n,h,a}}). Then, if $(n,h,a)=(N,s,2 \Re{s} + 1)$, $\Phi_N(x) = \tilde{P}_N (x)$ and if $(n,h,a)=(N-1,s+1, 2 \Re{s} + 2)$, $\Phi_N(x) = Q_N (x)$. Moreover, note that $\Phi_N$ is a product of $C^\infty$ functions on $\mathds{R}$.

Now, for $\delta \in \{-1,1\}$:
\begin{equation*}
\log (1+\delta i N x) = \log ( 1 - \delta i/Nx) + \log (N|x|) + i \frac{\pi}{2} \delta \textrm{Sgn}(x),
\end{equation*}
because both sides of the equality have an imaginary part in $(-\pi, \pi)$  and their exponentials are equal. Hence,
\begin{align*}
& \left( \frac{-s+N}{2} - (N-n) \right) \log (1 + iNx) + \frac{- \overline{s}-N}{2} \log (1-iNx) \\
=& \left( \frac{-s+N}{2} - (N-n) \right) \log (1 - i/Nx) + \frac{- \overline{s}-N}{2} \log (1+i/Nx)\\
&  - (\Re{s} + (N-n)) \log (N|x|)  + n i \pi \textrm{Sgn} (x)/2 + \pi \Im{s} \textrm{Sgn} (x)/2.
\end{align*}
This implies:
\begin{align}
\Phi_N(x) =& D(N,n,s) (-i)^n (1 + iNx)^{(-s+N)/2 - (N-n)} (1-iNx) ^{(-\overline{s} - N)/2} F_{n,h,a}(Nx) \nonumber\\
=& D(N,n,s) (-i)^n (N|x|)^{-\Re{s} - (N-n)} e^{ni \pi \textrm{Sgn} (x)/2} e^{\pi \Im{s} \textrm{Sgn}(x)/2}\nonumber\\
&(1 - i/Nx)^{(N-s)/2 - (N-n)} (1+i/Nx)^{(-\overline{s} - N)/2} F_{n,h,a}(Nx)\nonumber \\
=& D(N,n,s) (\textrm{Sgn} (x))^n (2N)^{- \Re{s} - (N-n)} (2/|x|)^{\Re{s} + N-n}
e^{\pi \Im{s} \textrm{Sgn}(x)/2}\nonumber\\
&(1 - i/Nx)^{(N-s)/2 - (N-n)} (1+i/Nx)^{(-\overline{s} - N)/2} F_{n,h,a}(Nx)\nonumber \\
=& D'(N,s) (\textrm{Sgn} (x))^N e^{\pi \Im{s} \textrm{Sgn}(x)/2} (2/x)^{N-n} (2/|x|)^{\Re{s}}\nonumber\\
& (1 - i/Nx)^{(N-s)/2 - (N-n)} (1+i/Nx)^{(-\overline{s} - N)/2} F_{n,h,a}(Nx),\label{Phi written out}
\end{align}
where for $s$ fixed,
\begin{equation}\label{D'(N,s)}
D'(N,s) = D(N,n,s) (2N)^{- \Re{s} - (N-n)} = \left( \frac{\Gamma(2 \Re{s} + N +1)}{N^{2 \Re{s} + 1}
\Gamma(N)} \right)^{1/2}.
\end{equation}
This tends to $1$ when $N$ goes to infinity. In particular $D'(N,s)$ can be bounded by some $C(s)$, not depending on $N$. We investigate all the terms in \eqref{Phi written out}  separately in the following.

Let $G$ be the function defined by:
\begin{equation*}
G(y) := (1 - iy/N)^{(N-s)/2 - (N-n)} (1+iy/N)^{(-\overline{s} - N)/2}.
\end{equation*}
This function is $C^{\infty}$ on $\mathds{R}$ and one has:
\begin{align*}
G^{(p)} (y)  =& G(y)  \sum_{q= 0}^p C(p,q) (i/N)^q (-i/N)^{p-q} (-(N-s)/2 + N-n)_q\\
& ((N+\overline{s})/2)_{p-q} (1-iy/N)^{-q} (1+iy/N)^{-(p-q)}.
\end{align*}
For $s$, $y$, $p$ and $N-n \in \{0,1\}$ fixed, the last sum is dominated by some constant $C(s,p)$ only depending on $s$ and $p$ and tends to $(-i)^p$, as $N\rightarrow\infty$. Moreover, $G(y)$ tends to $e^{-iy}$, and
\begin{equation*}
G(y) = \left( \frac{1-iy/N}{1+iy/N} \right)^{(N- i \Im{s})/2} (1-iy/N)^{-(N-n)} (1+y^2/N^2)^{-\Re{s}/2}.
\end{equation*}
A simple computation, yields the following:
\begin{equation*}
|G(y)| \leq C(s) \left(1+\frac{y^2}{N^2}\right)^{- \Re{s}/2} \leq C(s) (1+y^2)^{1/4}.
\end{equation*}
This implies that $G^{(p)} (y)$ tends to $(-i)^p e^{-iy}$ when $N$ goes to infinity,
and that
\begin{equation*}
\left|G^{(p)} (y)\right| \leq C(s,p) (1+y^2)^{1/4}.
\end{equation*}
Now, for all $f$ in $C^{\infty} (\mathds{R})$, the function $g$ defined by $x\mapsto f(1/x)$ is in $C^\infty(\mathds{R}^\ast)$, and there exist universal integers $(\mu_{p,k})_{p \in \mathds{N}_0, 0 \leq k \leq p}$, such that $\mu_{p,0} = 0$ for all $p \geq 1$, and for $p\in\mathds{N}_0$,
\begin{equation*}
g^{(p)} (x) = \sum_{k=0}^p \frac{\mu_{p,k}}{x^{p+k}} f^{(k)} (1/x).
\end{equation*}
Applying this formula to the functions $G$ and $y \rightarrow e^{-iy}$, one obtains the following pointwise convergence (for $x \neq 0$):
\begin{equation}\label{convergence of G}
\frac{d^p}{dx^p} \left[(1 - i/Nx)^{(N-s)/2 - (N-n)} (1+i/Nx)^{(-\overline{s} - N)/2} \right]
\underset{N \rightarrow \infty}{\longrightarrow} \frac{d^p}{dx^p} (e^{-i/x})
\end{equation}
with, for $|x| \geq x_0 > 0$,
\begin{equation}\label{bound on convergence of G}
\left|\frac{d^p}{dx^p} \left[(1 - i/Nx)^{(N-s)/2 - (N-n)} (1+i/Nx)^{(-\overline{s} - N)/2} \right]\right|
\leq \frac{C(x_0,s,p)}{|x|^{p + \mathds{1}_{p > 0} }}.
\end{equation}
Recall that by Lemma \ref{derivatives F_{n,h,a}}, one has the convergence
\begin{equation}\label{convergence of F}
\frac{d^p}{dx^p} (F_{n,h,a} (Nx))  \underset{N \rightarrow \infty}{\longrightarrow}
\frac{d^p}{dx^p} ( \,_1F_1  [h,a;2i/x]),
\end{equation}
and the bound
\begin{equation}\label{bound on convergence of F}
\left| \frac{d^p}{dx^p} (F_{n,h,a} (Nx)) \right| \leq \frac{C(x_0, h,a,p)}{|x|^{p + \mathds{1}_{p > 0}}}
\leq \frac{C(x_0,s,p)}{|x|^{p + \mathds{1}_{p > 0}}},
\end{equation}
since $(h,a)$ only depends on $s$ in the relevant cases (see the beginning of the proof).
Moreover,
\begin{equation}\label{local bound}
\left| \frac{d^p}{dx^p} \left[ (2/x)^{N-n} (2/|x|)^{ \Re{s}} \right]\right| \leq \frac{C(s,p)}{|x|^{ \Re{s}+(N-n)+p}}.
\end{equation}
We can now give the derivatives of $\Phi_N$, using \eqref{Phi written out}. One has for $p\geq0$:
\begin{align*}
( \textrm{Sgn} (x))^N \frac{d^p}{dx^p} (\Phi_N(x)) =& D'(N,s) e^{\pi \Im{s} \textrm{Sgn} (x)/2}\\
&\sum_{q_1 + q_2 + q_3 = p} \frac{p!}{q_1! q_2 ! q_3!} \frac{d^{q_1}}{dx^{q_1}} \left[ (2/x)^{N-n}
(2/|x|)^{\Re{s}} \right]\\
& \frac{d^{q_2}}{dx^{q_2}} \left[ (1-i/Nx)^{(N-s)/2 - (N-n)}
(1+i/Nx)^{(-N-\overline{s})/2} \right] \\
& \frac{d^{q_3}}{dx^{q_3}} \left[ F_{n,h,a} (Nx) \right].
\end{align*}
By \eqref{D'(N,s)}, \eqref{convergence of G} and \eqref{convergence of F}, whenever $s$, $x$ and $N-n\in\{0,1\}$ are fixed, this expression tends to
\begin{align*}
e^{\pi \Im{s} \textrm{Sgn} (x)/2} & \sum_{q_1 + q_2 + q_3 = p} \frac{p!}{q_1! q_2 ! q_3!} \frac{d^{q_1}}{dx^{q_1}} \left[ (2/x)^{N-n}
(2/|x|)^{\Re{s}} \right]  \\
& \frac{d^{q_2}}{dx^{q_2}} \left[ e^{-i/x} \right]
 \frac{d^{q_3}}{dx^{q_3}} \left( \,_1F_1[h,a;2i/x] \right),
\end{align*}
for $N \rightarrow \infty$. But this is precisely the $p$-th derivative of $\tilde{P}$ at $x$ if $\Phi_N = \tilde{P}_N$, and the $p$-th derivative of $Q$ at $x$ if $\Phi_N = Q_N$. Moreover, for $|x| \geq x_0 > 0$, one easily obtains the bound
\begin{equation*}
\left|\frac{d^p}{dx^p} (\Phi_N(x)) \right| \leq \frac{C(x_0,s,p)}{|x|^{\Re{s} + (N-n) +p}},
\end{equation*}
using \eqref{D'(N,s)}, \eqref{bound on convergence of G}, \eqref{bound on convergence of F} and \eqref{local bound}. This completes the proof of the Lemma.
\end{proof}

\begin{lemma} \label{fg}
Let $f$ and $g$ be two functions which are $C^{\infty}$ from $\mathds{R}^\ast$ to $\mathds{R}$. We define the function
$\phi$ from $(\mathds{R}^\ast)^2$ to $\mathds{R}$ by
\begin{equation*}
\phi(x,y) := \frac{f(x) g(y)- g(x) f(y)}{x-y},
\end{equation*}
for $x \neq y$, and
\begin{equation*}
\phi(x,x) := f'(x) g(x) - g'(x)f(x).
\end{equation*}
Then, $\phi$ is $C^{\infty}$ on $(\mathds{R}^*)^2$ and for all $p, q \in\mathds{N}_0$:
\newline
(a) If $x \neq y$:
\begin{equation*}
\frac{\partial^{p+q}\phi}{\partial x^p \partial y^q} = \sum_{k=0}^p \sum_{l=0}^q
C_p^k C_q^l \frac{ f^{(k)} (x) g^{(l)} (y) - g^{(k)} (x) f^{(l)} (y) }{(x-y)^{p+q-k-l+1}} (-1)^{p-k}(p+q-k-l)!.
\end{equation*}
(b) If $x$ and $y$ have same sign:
\begin{align*}
\frac{\partial^{p+q}\phi}{\partial x^p \partial y^q}=& \sum_{k=0}^q C_q^k \left[ g^{(q-k)} (y)
\int_0^1 f^{(k+p+1)} (y+\theta(x-y)) \theta^p (1- \theta)^k d \theta \right. \\
& \left.- f^{(q-k)} (y) \int_0^1 g^{(k+p+1)} (y+\theta(x-y)) \theta^p (1- \theta)^k d \theta \right].
\end{align*}
\end{lemma}

\begin{proof}
(a) By induction, one proves that for all $p,q\in\mathds{N}_0$, and for $x,y\in\mathds{R}$ distincts and different from zero, it is possible to take, in a neighborhood of $(x,y)$, $p$ derivatives of $\phi$ with respect to $x$ and $q$ derivatives  of $\phi$ with respect to $y$, in any order, with a result equal to the expression given in the statement of the Lemma. This implies the existence and the continuity of all partial derivatives of $\phi$ in $(\mathds{R}^\ast)^2 \backslash \{ (x,x), x \in \mathds{R}^\ast \}$. Therefore, $\phi$ is
$C^{\infty}$ in this open subset of $(\mathds{R^\ast})^2$. \newline
(b) With the same method as in (a), we obtain that $\phi$ is $C^{\infty}$ on $(\mathds{R}_-^*)^2 \cup (\mathds{R}_+^*)^2$. The only technical issues are the continuity and the derivation under the integral. These can easily be justified by the boundedness of the derivatives of $f$ and $g$ in any compact set of
$\mathds{R}^*$.
\end{proof}

\begin{proposition} \label{convergence of K}
Let $x, y \in \mathds{R}^*$ and let $\Re{s} > -1/2$. Then $K_{[N]}$ and $K_\infty$ are $C^{\infty}$ in $(\mathds{R}^\ast)^2$ and for all $p, q \in \mathds{N}_0$,
\begin{equation*}
(\textrm{Sgn} (xy))^N \frac{\partial^{p+q}}{\partial x^p \partial y^q} K_{[N]} (x,y) \underset{N \rightarrow \infty}{\longrightarrow} \frac{\partial^{p+q}}{\partial x^p \partial y^q} K_\infty (x,y).
\end{equation*}
Moreover, for any $x_0>0$, and  $|x|, |y| \geq x_0 > 0$:
\begin{equation*}
\left| \frac{\partial^{p+q}}{\partial x^p \partial y^q} K_{[N]} (x,y)\right| \leq \frac{C(x_0,s,p,q)} {|x|^{\Re{s} + p + 1} |y|^{\Re{s} + q+ 1}}.
\end{equation*}
\end{proposition}

Note that the pointwise convergence in the case $p=q=0$ corresponds to the convergence result for the kernels given by Borodin and Olshanski \cite{Borodin-Olshanski}.

\begin{proof}
One has
\begin{align*}
& (\textrm{Sgn} (xy))^N K_{[N]} (x,y) = \frac{1}{2\pi} \frac{\Gamma(s+1) \Gamma(\overline{s}+1)}
{\Gamma (2 \Re{s} + 1) \Gamma (2 \Re{s} + 2) } \\
&\frac{ (\textrm{Sgn} (x))^N \tilde{P}_N(x)
(\textrm{Sgn} (y))^N Q_N (y) - (\textrm{Sgn} (y))^N \tilde{P}_N(y)
(\textrm{Sgn} (x))^N Q_N (x) }{x-y}
\end{align*}
for $x \neq y$, and
\begin{equation*}
K_{[N]} (x,x) =  \frac{1}{2\pi} \frac{\Gamma(s+1) \Gamma(\overline{s}+1)} {\Gamma (2 \Re{s} + 1) \Gamma (2 \Re{s} + 2) } ( \tilde{P}_N'(x) Q_N (x) - Q_N'(x) \tilde{P}_N(x)),
\end{equation*}
with $\tilde{P}_N$ and $Q_N$ defined in Lemma \ref{derivatives P_N Q_N}. Recall the definition of $K_\infty$ in \eqref{definition of K infinity} and \eqref{definition of K infinity for x=x}.
Now, $\tilde{P}_N$, $Q_N$, $\tilde{P}$ and $Q$ are in $C^{\infty} (\mathds{R}^\ast)$ (see Lemma \ref{derivatives P_N Q_N}) and hence, by Lemma \ref{fg}, $K_{[N]}$ and $K_\infty$ are in $C^{\infty} ((\mathds{R}^\ast)^2)$.

Moreover, by Lemma \ref{derivatives P_N Q_N}, the derivatives of $x \mapsto \textrm{Sgn}^N (x) \tilde{P}_N(x)$ and $x \mapsto \textrm{Sgn}^N (x) Q_N(x)$ converge pointwise to the corresponding derivatives of $\tilde{P}$ and $Q$. Considering, for $x \neq y$, the expression (a) of Lemma \ref{fg}, and for $x=y$, the expression (b), one easily deduces the pointwise convergence of the derivatives of $(x,y) \mapsto (\textrm{Sgn} (xy))^N K_{[N]}(x,y) $ towards the corresponding derivatives of $K_\infty$.

Finally, the bounds given in the statement of the Lemma can be obtained from the bounds of the derivatives of $\tilde{P}_N$ and $Q_N$, given in Lemma \ref{derivatives P_N Q_N}, and by applying the formula (a) of Lemma \ref{fg} if $xy < 0$ or $\max(|x|,|y|) > 2 \min(|x|,|y|)$ (which implies $|x-y| \geq \max(|x|,|y|)/2$), or the formula (b) if $xy > 0$ and $\max(|x|,|y|) \leq 2 \min(|x|,|y|)$.
\end{proof}

Summarizing, we have:

\begin{proposition}\label{K_N and K_infinity satisfy the crucial conditions}
Let $s$ be such that $\Re{s}>-\frac{1}{2}$. Then, the restriction to $\mathds{R}_+^\ast$ of the scaled kernel $K_{[N]}$ and the kernel $K_\infty$  satisfy the conditions of Lemma \ref{key estimates for convergence theorem}. Moreover, for all $p,q\in\mathds{N}_0$, the partial derivatives
\begin{equation*}
\textrm{Sgn}(xy)^N\frac{\partial^{p+q}}{\partial x^p \partial y^q}K_{[N]}(x,y)
\end{equation*}
converge pointwise to the corresponding partial derivatives of $K_\infty(x,y)$.
\end{proposition}

\begin{proof}
This follows immediately from Proposition \ref{convergence of K} and the fact that these kernels are real symmetric and positive because they are kernels of determinantal processes on the real line (see remark \ref{K_infty defines a det process} for the kernel $K_\infty$).
\end{proof}

The next step  is to analyze the convergence of the Fredholm determinant of $K_{N,J}$ and its derivatives to the corresponding derivatives of the Fredholm determinant of $K_{\infty,J}$, for $J=(t,\infty)$, $t>0$.

\begin{lemma} \label{derivation F(t,x)}
Let $F$ be a function defined from $(\mathds{R}_+^*)^{k+1}$ to $\mathds{R}$, for some $k \in \mathds{N}$. We suppose that $F$ is in $C^1$, and that there exists, for some $\alpha>1$ and for all $x_0>0$, a bound of the form
\begin{equation*}
\left|F(t,x_1,x_2,...,x_k)\right| + \left| \frac{\partial}{\partial t} F(t,x_1,x_2,...,x_k)\right| \leq \frac{C(x_0)}
{(x_1...x_k)^{\alpha}},
\end{equation*}
for all $t, x_1,...,x_k \geq x_0$. Then, the integrals involved in the definitions of the following two functions from $\mathds{R}_+^*$ to $\mathds{R}$ are absolutely convergent:
\begin{equation*}
H_0 : t \mapsto \int_{(t, \infty)^k} F(t,x_1,\ldots,x_k) dx_1\ldots dx_k,
\end{equation*}
and
\begin{align*}
H_1 : t \mapsto & \int_{(t, \infty)^k} \frac{\partial}{\partial t} F(t,x_1,\ldots,x_k) dx_1\ldots dx_k \\
& - \sum_{l=1}^k \int_{(t, \infty)^{k-1}} F(t,x_1,\ldots,x_{l-1},t,x_{l+1},\ldots,x_k) dx_1\ldots dx_{l-1}
dx_{l+1}\ldots dx_k.
\end{align*}
Moreover, the first derivative of $H_0$ is continuous and equal to $H_1$.
\end{lemma}

\begin{proof}
Due to the bound given in the Lemma, it is clear that all the integrals in the definition of $H_0$ and $H_1$ are absolutely convergent. Therefore, for $0<t<t'$, we can use Fubini's Theorem in order to compute the integral
\begin{equation*}
\int_{t}^{t'} H_1(u) du.
\end{equation*}
Straightforward computations show that this integral is equal to $H_0(t') - H_0(t)$. Hence, if we prove that $H_1$ is continuous, we are done. Now, let $t> x_0> 0$. For $t' > x_0$, one has
\begin{align*}
& |H_1(t') - H_1(t)|  \leq \int_{(x_0, \infty)^k} \left|\frac{\partial}{\partial t'} F(t',x_1,\ldots,x_k)
\mathds{1}_{\{x_1,\ldots,x_k > t'\}}  \right.\\
& \left. - \frac{\partial}{\partial t} F(t,x_1,\ldots,x_k)
\mathds{1}_{\{x_1,\ldots,x_k > t\}}\right|  dx_1\ldots dx_k \\
&+ \sum_{l=1}^k \int_{(x_0, \infty)^{k-1}} \left|F(t',x_1,\ldots,x_{l-1},t',x_{l+1},\ldots,x_k)
\mathds{1}_{\{x_1,\ldots,x_{l-1}, x_{l+1},\ldots x_k > t'\}} \right.\\
&\left. -  F(t,x_1,\ldots,x_{l-1},t,x_{l+1},\ldots,x_k)
\mathds{1}_{\{x_1,\ldots,x_{l-1}, x_{l+1},\ldots x_k > t\}} \right| dx_1\ldots dx_{l-1}
dx_{l+1}\ldots dx_k.
\end{align*}
All the terms inside the integrals converge to zero almost everywhere when $t'\rightarrow t$ (more precisely, whenever the minimum of the $x_j$'s is different from $t$). Hence, by dominated convergence, $|H_1 (t')- H_1(t)|$ tends to zero when $t'\rightarrow t$.
\end{proof}

\begin{lemma} \label{Fredholm}
Let $K$ be a function satisfying the conditions of Lemma \ref{key estimates for convergence theorem}. Then, using the notation of that Lemma,
\begin{equation*}
\sum_{k \geq 1} \frac{1}{k!} \int_{(t, \infty)^k} \rho_k(x_1,\ldots,x_k) dx_1\ldots dx_k < \infty
\end{equation*}
for all $t> 0$. Moreover, the Fredholm determinant  $F$, from $\mathds{R}_+^*$ to $\mathds{R}$, defined in (\ref{Rubin})
is in $C^3$, and its derivatives are given by
\begin{align*}
F'(t) &= \sum_{k \geq 0} \frac{(-1)^k}{k!} \int_{(t, \infty)^k} \rho_{k+1}(t,x_1,\ldots,x_k) dx_1\ldots dx_k,\\
F''(t) &= \sum_{k \geq 0} \frac{(-1)^k}{k!} \int_{(t, \infty)^k} \frac{\partial}{\partial t}
\rho_{k+1}(t,x_1,\ldots,x_k) dx_1\ldots dx_k,\\
F'''(t) &= \sum_{k \geq 0} \frac{(-1)^k}{k!} \int_{(t, \infty)^k} \frac{\partial^2}{\partial t^2}
\rho_{k+1}(t,x_1,\ldots,x_k) dx_1\ldots dx_k,
\end{align*}
where all the sums and the integrals above are absolutely convergent.
\end{lemma}

\begin{proof}
For $k \geq 1$, we define $F_k$ by
\begin{equation*}
F_k(t) =  \frac{(-1)^k}{k!} \int_{(t, \infty)^k} \rho_k(x_1,\ldots,x_k) dx_1\ldots dx_k.
\end{equation*}
The integral is finite because of the bounds given in Lemma \ref{key estimates for convergence theorem}. By the same bounds, one can apply Lemma \ref{derivation F(t,x)} three times, to obtain that $F_k$ is in $C^3$, with the derivatives given by
\begin{align*}
F'_k(t) &= \frac{(-1)^{k-1}}{(k-1)!} \int_{(t, \infty)^{k-1}} \rho_{k}(t,x_1,\ldots,x_{k-1}) dx_1\ldots dx_{k-1},\\
F''_k(t) &= \frac{(-1)^{k-1}}{(k-1)!} \int_{(t, \infty)^{k-1}} \frac{\partial}{\partial t}\rho_{k}(t,x_1,\ldots,x_{k-1}) dx_1\ldots dx_{k-1},\\
F'''_k(t) &= \frac{(-1)^{k-1}}{(k-1)!} \int_{(t, \infty)^{k-1}} \frac{\partial^2}{\partial t^2}\rho_{k}(t,x_1,\ldots,x_{k-1}) dx_1\ldots dx_{k-1},
\end{align*}
where again all the integrals are absolutely convergent by Lemma \ref{key estimates for convergence theorem}. Note that we use \eqref{partial rho equal 0} to calculate the derivatives above. Moreover, for $p \in \{0,1,2,3\}$, \eqref{main bound for K} gives the following bound for any $x_0>0$:
\begin{equation*}
\underset{t \geq x_0}{\sup} \, |F^{(p)}_k(t)| \leq \frac{(C(x_0))^k}{(k-1)!}.
\end{equation*}
Using dominated convergence, we have that the sum
\begin{equation*}
\sum_{k \geq 1} F_k(t)
\end{equation*}
is absolutely convergent, and that its $p$-th derivative, $p\in\{0,1,2,3\}$ with respect to $t$ is continuous
and given by the absolutely convergent sum
\begin{equation*}
\sum_{k \geq 1} F^{(p)}_{k} (t).
\end{equation*}
\end{proof}

\subsection{$\theta_\infty$ is well defined}\label{Rubin3}In order to prove that $\theta_\infty$ is well defined,
we need to prove that $F_{\infty}(t)$ never vanishes for $t > 0$ (recall from Remark \ref{t must be bigger than 0} that the range of the largest eigenvalue is $\mathds{R}_+^*$). We
note that $F_{\infty}(t)$ is the Fredholm determinant of the
restriction of the operator $K_{\infty}$ to the space
$L^2((t,\infty))$,
which can also be seen as the operator on  $L^2((t_0,\infty))$
with kernel $(x,y) \rightarrow
K_{\infty}(x,y) \, \mathds{1}_{x,y>t}$, for some $t_0$ such that
$t > t_0>0$.  This operator is positive, and we
recall that it is a trace class operator, since:
$$\int_{(t,\infty)} K_{\infty}(x,x) \, dx < \infty.$$
Therefore, the Fredholm determinant of this operator is given by the convergent product
of $1- \lambda_j$, where $(\lambda_j)_{j \in \mathds{N}}$ is the
decreasing sequence of its (positive) eigenvalues, with multiplicity.
This implies that the determinant is zero if and only if  $1$ is an eigenvalue of
the operator: hence, we only need to prove that this is not the case.
Indeed, if $1$ is an eigenvalue, there exists $f \neq 0$ in $L^2
((t_0,\infty))$
such that for almost all $x \in (t_0, \infty)$:
$$f(x)  = \mathds{1}_{x > t} \, \int_t^{\infty}  K_{\infty}(x,y) \,
f(y) \, dy.$$
Hence $f (x) = 0$ for almost every $x \leq t$, and
$$ f = p_{(t,\infty)} K_{\infty, ( t_0,\infty)} f$$
in  $L^2((t_0,\infty))$, where $K_{\infty, ( t_0,\infty)} $ is the
operator on this space, with kernel $K_{\infty}$, and $p_{(t,\infty)}$
is the projection on the space of functions supported by $(t,
\infty)$. Now, if we denote $ g :=  K_{\infty, ( t_0,\infty)} f$,
$$  ||g||_{L^2 ((t_0,\infty))}^2
= \int_{t_0}^{\infty}  \int_{t_0}^{\infty} \int_{t_0}^{\infty}  K_{\infty}(x,y) \,
K_{\infty}(x,z) \, f(y) \, f(z) \, dx \, dy \, dz.$$
By dominated convergence, one can check that $||g||_{L^2 ((t_0,\infty))}^2
$ is the limit of
$$\int_{t_0}^{\infty}  \int_{t_0}^{\infty} \int_{t_0}^{\infty}  K_{[N]}(x,y) \,
K_{[N]}(x,z) \, f(y) \, f(z) \, dx \, dy \, dz$$
when $N$ goes to infinity. This expression is equal to
$ || p_{(t_0,\infty)} K_{[N]} \tilde{f}||_{L^2(\mathds{R})}^2$, and hence,
smaller than or equal to $ ||K_{[N]} \tilde{f}||_{L^2(\mathds{R})}^2$,
where the operators $p_{(t_0,\infty)}$ and  $K_{[N]}$ act on $L^2(\mathds{R})$,
and where  $\tilde{f}$ is equal to $f$ on $(t_0, \infty)$ and equal to
zero on $ (-\infty, t_0]$. Now, $K_{[N]}$ (as $K_N$) is an orthogonal projector
on $L^2(\mathds{R})$ (with an $N$-dimensional image), hence,
$||K_{[N]} \tilde{f}||_{L^2(\mathds{R})} \leq ||\tilde{f}||_{L^2(\mathds{R})}$.
This implies:
$$ ||g||_{L^2 ((t_0,\infty))} \leq
||f||_{L^2 ((t_0,\infty))}.$$
Now, with obvious notation:
\begin{align*}
||g||_{L^2 ((t_0,\infty))}^2 & = ||p_{(t,\infty)} g||_{L^2 ((t_0,\infty))}^2 +
||p_{(t_0,t]} g||_{L^2 ((t_0,\infty))}^2  \\ & = ||f||_{L^2 ((t_0,\infty))}^2 +
||p_{(t_0,t]} g||_{L^2 ((t_0,\infty))}^2
\end{align*}
since $ f = p_{(t,\infty)} g$. By comparing the last two equations, one deduces that
$$||p_{(t_0,t]} g||_{L^2 ((t_0,\infty))}^2  = 0,$$
which implies that $g$ is supported by $(t, \infty)$, and
$$f = p_{(t,\infty)} g = g =  K_{\infty, ( t_0,\infty)} f.$$
Hence, $K_{\infty, ( t_0,\infty)} f$ (equal to $f$), takes the value zero a.e. on the interval $(t_0,t)$.
Since $f$ is different from zero, one easily deduces a contradiction from the following Lemma:
\begin{lemma} \label{analytic}
Let $f$ be a function in $L^2 ((t, \infty))$ for some $t>0$. Then the function $g$ from $\mathds{R}_+^*$ to
$\mathds{R}$, defined by:
$$g(x) = \int_{t}^ {\infty} K_{\infty} (x,y) \, f(y) \, dy$$
is analytic on $\{z\in\mathbb{C};\;\Re(z)>0\}$.
\end{lemma}
\begin{proof}
It is sufficient to prove that for all $x_0$ such that $0<x_0<t/2$,
$g$ can be extended to a holomorphic function on the set $H_{x_0} := \{x \in \mathds{C}; \Re(x) > x_0
\}$. Let $(\epsilon, h,a)$ be equal to $(0, s, 2\Re(s)+1)$ or $(1,s+1, 2\Re(s) + 2)$, and $\Phi$
equal to $\tilde{P}$ in the first case, $Q$ in the second case. One has for $x \in \mathds{R}_+^*$:
$$\Phi(x) = \left( \frac{2}{x} \right)^{\Re(s) + \epsilon} e^{-i/x} e^{\pi \Im (s)/2}
\,_1F_1[h,a;2i/x].$$
$\Phi$ can easily be extended to $H_{x_0}$: for the first factor, one can use the standard extention of the logarithm
(defined on $\mathds{C} \backslash \mathds{R}_-$), and the last factor is a hypergeometric
series which is uniformly convergent on $H_{x_0}$. Moreover, it is easy to check (by using dominated convergence
for the hypergeometric factor), that this extension of $\Phi$ is holomorphic with derivative:
\begin{align*}
&\Phi'(x)=\\
&e^{\pi\Im{s}/2}\left(\frac{2}{x}\right)^{\Re{s}+\epsilon}e^{-i/x}\\
&\cdot\left[\frac{-(\Re{s}+\epsilon)}{x}{}_1F_1[h,a;2i/x]+\frac{i}{x^2}{}_1F_1[h,a;2i/x]\right.\\
&\left.-\sum_{k=0}^\infty\frac{(h)_k(2i)^kk}{(a)_kk!}\left(\frac{1}{x}\right)^{k+1}\right].
\end{align*}
With these formulae, one deduces the following bounds, available on the whole set $H_{x_0}$:
$$|\Phi(x)| \leq \frac{C(x_0,s)}{|x|^{\Re (s) + \epsilon}},$$
$$|\Phi'(x)| \leq \frac{C(x_0,s)}{|x|^{\Re(s) + \epsilon + 1}}.$$
Now, let us fix $y \in (t, \infty)$. Recall that for $x \in \mathds{R}_+^* \backslash \{y\}$:
\begin{equation} \label{www}
K_{\infty} (x,y) = \frac{1}{2\pi} \frac{\Gamma(s+1) \Gamma(\overline{s} + 1) }
{\Gamma(2 \Re{s} + 1) \Gamma(2 \Re{s} + 2)} \frac{\tilde{P} (x) Q(y) - Q(x) \tilde{P} (y) }{x-y}.
\end{equation}
This formula is meaningful for all $x \in H_{x_0} \backslash \{y\}$ and gives an analytic
continuation of $x \mapsto K_{\infty}(x,y)$ to this set. Now, for $x>x_0$, one also has the formula:
\begin{equation*}
K_{\infty} (x,y) = \frac{1}{2\pi} \frac{\Gamma(s+1) \Gamma(\overline{s} + 1) }
{\Gamma(2 \Re{s} + 1) \Gamma(2 \Re{s} + 2)} \mathds{E} \left[
\tilde{P}' (Z) Q(y) - Q'(Z) \tilde{P} (y) \right],
\end{equation*}
where $Z$ is a uniform random variable on the segment $[x,y]$. By the bounds obtained for
$\Phi$ and $\Phi'$, one deduces that the continuation of $x \mapsto K_{\infty}(x,y)$
to the set $H_{x_0} \backslash \{y\}$ is bounded in the neighborhood of $y$, and hence can
be extended to $H_{x_0}$. By construction, this extension coincides with $K_{\infty} (x,y)$ for
$x \in (x_0, \infty) \backslash \{y\}$, and in fact it coincides on the whole interval
$(x_0, \infty)$, since $K_{\infty} (x,y)$ tends to $K_{\infty} (y,y)$ when $x$ is real and tends to $y$.
In other words, we have constructed an extension of $x \mapsto K_{\infty} (x,y)$ which is holomorphic
on $H_{x_0}$. Now, let us take $x \in H_{x_0}$ such that $|x-y| \geq y/2$, which implies that
$|x-y| \geq C (|x|+y)$ for a universal constant $C$. By using this inequality and the bounds on
$\tilde{P}$ and $Q$, one obtains:
$$|K_{\infty} (x,y)| \leq \frac{C(s,x_0)}{|xy|^{\Re(s) + 1}}.$$
By taking the derivative of the equation \eqref{www}, one obtains the bound (again for
$x \in H_{x_0}$ and $|x-y| \geq y/2$):
$$\left| \frac{\partial}{\partial x} K_{\infty} (x,y) \right| \leq \frac{C(s,x_0)}{|x|^{\Re(s) + 2} y^{\Re(s)+1} }.$$
Now, the maximum principle implies that the condition $|x-y| \geq y/2$ can be removed
in the  last two bounds. By using these bounds, Cauchy-Schwarz inequality and dominated
convergence, one deduces that the function:
$$x \mapsto \int_{t}^{\infty} K_{\infty}(x,y) \, f(y) \, dy$$
is well defined on the set $H_{x_0}$, and admits a derivative, given by the formula:
$$x \mapsto \int_{t}^{\infty} \left( \frac{\partial}{\partial x} K_{\infty} (x,y)  \right) \, f(y) \, dy.$$
%\begin{equation}\label{definition of K infinity for x=x}
%K_{\infty}(y,y)=\frac{1}{2\pi} \frac{\Gamma(s+1) \Gamma(\overline{s} + 1) }
%{\Gamma(2 \Re{s} + 1) \Gamma(2 \Re{s} + 2)}(\tilde{P}'(y)Q(y)-Q'(y)\tilde{P}(y)),
%\end{equation}
\end{proof}

\subsection{Proof of Theorem \ref{Fredholm2}}

Note that by Proposition \ref{K_N and K_infinity satisfy the crucial conditions}, $K_{[N]}$ and $K_\infty$ satisfy the conditions of Lemma \ref{key estimates for convergence theorem}. For $k, N \in \mathds{N}$,  let $\rho_{k,N}$ be the $k$-correlation function associated with $K_{[N]}$ and $\rho_{k, \infty}$ the $k$-correlation function associated with $K_\infty$. By Lemma \ref{Fredholm}, $F_N$ is well defined for $N \in \mathds{N} \cup\{ \infty \}$, and $C^3$. The explicit expressions of $F_N$ and $F_\infty$ and their derivatives are given in Lemma \ref{Fredholm} by replacing $\rho_k$ by $\rho_{k,N}$ and $\rho_{k,\infty}$ respectively. Now, for $k \geq 1$, all the partial derivatives of any order of $\rho_{k,N}$ converge pointwise to the corresponding derivatives of $\rho_{k, \infty}$ when $N$ goes to infinity. This is due to the explicit expression of $\rho_{k,N}$ as a determinant and the convergence given by Proposition \ref{K_N and K_infinity satisfy the crucial conditions}. Moreover, by that same Proposition, there exists $\alpha>1/2$ only depending on $s$ such that
\begin{equation*}
\left|\frac{\partial^{p}}{\partial x_1^p} \rho_{k,N} (x_1,\ldots,x_k)\right| \leq \frac{C(x_0,s)^k}{(x_1\ldots x_k)^{2 \alpha}},
\end{equation*}
for $p \in \{0,1,2\}$, and for all $x_1,...,x_k \geq x_0>0$. In particular, this bound is uniform with respect to $N$, and it is now easy to deduce the pointwise convergence of the derivatives of $F_N$ (up to order $3$), by dominated convergence.

\subsection{Proof of Theorem \ref{convergence of the N-solution to the infinity-solution}}
Theorem  \ref{convergence of the N-solution to the infinity-solution} follows immediately from Proposition \ref{convergence of the PVI to the PV equation} and the following Proposition:
\begin{proposition} \label{theta}
Let $s$ be such that $\Re{s} > -1/2$, and $F_N$, $N \in \mathds{N}$, and $F_{\infty}$ be as in Theorem \ref{Fredholm2}. Then, for $N \in \mathds{N} \cup\{ \infty\}$, the function $\theta_N$ from $\mathds{R}_+^*$ to $\mathds{R}$, defined by
\begin{equation*}
\theta_N(\tau) = \tau \frac{d}{d \tau} \log( F_N(\tau^{-1})),
\end{equation*}
is well defined and $C^2$. Moreover, for $p \in \{0,1,2\}$, the derivatives $\theta^{(p)}_N$ converge pointwise to $\theta^{(p)}_{\infty}$ (defined by (\ref{theta_infty})).
\end{proposition}

\begin{proof}
Recall that for $t > 0$, $F_N(t)$ is the probability that a random matrix of dimension $N$, following the generalized Cauchy weight \eqref{Full Hua-Pickrell}, has no eigenvalue in $(Nt, \infty)$. Therefore, $F_N(t)>0$, for any $t>0$. Similarly, $F_{\infty}(t)$ is the probability that the limiting determinantal process has no point in $(t, \infty)$, which  is also  different from zero for any $t>0$, as we proved in Subsection \ref{Rubin3}. Therefore, for all $N \in \mathds{N} \cup\{ \infty\}$, $\theta_N$ is well-defined and
\begin{equation*}
\theta_N(\tau) = - \frac{F'_N(\tau^{-1})}{\tau F_N(\tau^{-1})}.
\end{equation*}
Since $F_N$ is in $C^3$, $\theta_N$ is in $C^2$, for all $N\in\mathds{N}\cup\{\infty\}$, and one can give explicit expressions for $\theta_N$ and for its first two derivatives (see Lemma \ref{Fredholm}). It is now easy to deduce from these explicit expressions and the pointwise convergence of the first three derivatives of $F_N$ assured by Theorem \ref{Fredholm2}, the pointwise convergence for the first two derivatives of $\theta_N$, when $N \in \mathds{N}$ goes to infinity.
\end{proof}

\begin{remark}\label{infinity solution via Tracy-Widom}
Note that most probably, it is also possible to derive the fact that the kernel $K_\infty$ gives rise to a solution of the Painlev?-V equation \eqref{Painleve V for theta} directly by the methods of Tracy and Widom \cite{Tracy-Widom} in an analogous way then the one used to obtain the Painlev?-VI equation \eqref{ODE for sigma 2} in the finite $N$ case. In fact, the recurrence equations \eqref{General recurrence equation for phi and psi} in the infinite case are:
\begin{align*}
x^2P'(x)&=\left(-x\Re{s}+\frac{\Im{s}}{\Re{s}}\right)P(x)+\frac{|s|^2}{\Re{s}^2}\frac{1}{2\Re{s}+1}Q(x),\\
x^2Q'(x)&=-\left(2\Re{s}+1\right)P(x)-\left(-x\Re{s}+\frac{\Im{s}}{\Re{s}}\right)Q(x),
\end{align*}
where $P$ and $Q$ are as in the definition of $K_\infty$ in \eqref{definition of K infinity} and \eqref{definition of K infinity for x=x}. However, this method will has several drawbacks, as already mentioned in the introduction.
\end{remark}

\section{The convergence rate: proof of Theorem \ref{rate of convergence for large x}}\label{Rate of Convergence for the Distribution of the Largest Eigenvalue}

We first need the rate of convergence for the scaled kernel $K_{[N]}(x,y)=NK_N(Nx,Ny)$:

\begin{lemma}\label{rate of convergence of K_N}
Let $x,y>x_0>0$. Then there exists a constant $C(x_0,s)>0$ only depending on $x_0$ and $s\in\mathds{C}$ ($\Re{s}>-1/2$), such that
\begin{equation*}
\left|K_{[N]}(x,y)-K_\infty(x,y)\right|\leq\frac{1}{N}\frac{C(x_0,s)}{(xy)^{\Re{s}+1}}.
\end{equation*}
\end{lemma}
In the following proof, $C(a,b,\ldots)$ denotes a strictly positive constant only depending on $a,b,\ldots$ which may change from line to line.
\begin{proof}
Let $x,y>x_0$, $x\neq y$. Then, setting $C(s)=\left|\frac{1}{2\pi}\frac{\Gamma(s+1)\Gamma(\overline{s}+1)}{\Gamma(2\Re{s}+1)\Gamma(2\Re{s}+2)}\right|$, and using the notations from Lemma \ref{derivatives P_N Q_N}, we have
\begin{align}
&\left|K_{[N]}(x,y)-K_\infty(x,y)\right|=  \label{machin} \\
&C(s)\left|\frac{1}{x-y}\right|\left|\tilde{P}_N(x)Q_N(y)-\tilde{P}_N(y)Q_N(x)-(\tilde{P}(x)Q(y)-\tilde{P}(y)Q(x))\right| \nonumber \\
\leq& C(s)\left|\frac{1}{x-y}\right|\left\{\left|\tilde{P}_N(x)Q_N(y)-\tilde{P}(x)Q(y)\right|+\left|\tilde{P}_N(y)Q_N(x)-\tilde{P}(y)Q(x)\right|\right\} \nonumber \\
\leq& C(s)\left|\frac{1}{x-y}\right|\left\{\left|\tilde{P}_N(x)-\tilde{P}(x)\right|\left|Q_N(y)\right|+\left|Q_N(y)-Q(y)\right|\left|\tilde{P}(x)\right|\right. \nonumber \\
&\left.+\left|\tilde{P}_N(y)-\tilde{P}(y)\right|\left|Q_N(x)\right|+\left|Q_N(x)-Q(x)\right|\left|\tilde{P}(y)\right|\right\}.
\nonumber
\end{align}
Similarly, if $x,y>x_0$, it is easy to check (by using the fundamental Theorem of calculus) that
\begin{align}
&\left|K_{[N]}(x,y)-K_\infty(x,y)\right|\leq C(s) \mathds{E} \left[ \left|\tilde{P}'_N(Z)-\tilde{P}'(Z)\right|\left|Q_N(x)\right| \right. \label{truc} \\
&\left. +\left|Q_N(x)-Q(x)\right|\left|\tilde{P}'(Z)\right|
+\left|\tilde{P}_N(x)-\tilde{P}(x)\right|\left|Q'_N(Z)\right|+\left|Q'_N(Z)-Q'(Z)\right|\left|\tilde{P}(x)\right|\right]. \nonumber
\end{align}
where $Z$ is a uniform random variable in the interval $[x,y]$. \\
By using \eqref{machin} if $\max(x,y) \geq 2 \min(x,y)$ and \eqref{truc}
if $\max(x,y) < 2 \min(x,y)$, one deduces that the Lemma is proved, if we show that for $p\in\{0,1\}$,
\begin{equation}\label{P converges like 1/N}
\left|\tilde{P}_N^{(p)}(x)-\tilde{P}^{(p)}(x)\right|\leq\frac{1}{N}\frac{C(x_0,s,p)}{x^{p+\Re{s}}},
\end{equation}
and
\begin{equation}\label{Q converges like 1/N}
\left|Q_N^{(p)}(x)-
Q^{(p)}(x)\right|\leq\frac{1}{N}\frac{C(x_0,s,p)}{x^{p+1+\Re{s}}},
\end{equation}
Recall from \eqref{Phi written out}, the following function (note that $x>x_0>0$):
\begin{align*}
\Phi_N(x)&=D'(N,s)e^{\pi\Im{s}/2}\left(\frac{2}{x}\right)^{N-n}\left(\frac{2}{x}\right)^{\Re{s}}\\
\cdot&\left(1-\frac{i}{Nx}\right)^{(N-s)/2-(N-n)}\left(1+\frac{i}{Nx}\right)^{-(\overline{s}+N)/2}F_{n,h,a}(Nx),
\end{align*}
and let us define similarly:
\begin{equation*}
\Phi(x)=e^{\pi\Im{s}/2}\left(\frac{2}{x}\right)^{N-n}\left(\frac{2}{x}\right)^{\Re{s}}e^{-i/x}{}_1F_1\left[h,a;2i/x\right],
\end{equation*}
where $(n,h,a)=(N,s,2\Re{s}+1)$ and $\Phi_N(x)=\tilde{P}_N(x)$, or $(n,h,a)=(N-1,s+1,2\Re{s}+2)$ and $\Phi_N(x)=Q_N(x)$, for $N\in\mathds{N}^*$ (recall that $N-n=0$ in the first case and $N-n=1$ in the second case). It suffices to show that for $p\in\{0,1\}$, $|\Phi_N^{(p)}(x)-\Phi^{(p)}(x)|\leq\frac{C(x_0,s,p)}{Nx^{\Re(s)+1+p}}$ to deduce \eqref{P converges like 1/N} and \eqref{Q converges like 1/N}. Let us first investigate the case $p=0$:
\begin{align}\label{general domination for Phi}
&\left|\Phi_N(x)-\Phi(x)\right|\leq e^{\pi\Im{s}/2}\left(\frac{2}{x}\right)^{\Re{s}+(N-n)}\\
&\cdot\left\{\left|D'(N,s)-1\right|\left|\left(1-i/(Nx)\right)^{(N-s)/2-(N-n)}\left(1+i/(Nx)\right)^{-(N+\overline{s})/2}F_{n,h,a}(Nx)\right|\right.\nonumber\\
&+\left|\left(1-i/(Nx)\right)^{(N-s)/2-(N-n)}\left(1+i/(Nx)\right)^{-(N+\overline{s})/2}-e^{-i/x}\right|\left|F_{n,h,a}(Nx)\right|\nonumber\\
&\left.+\left|e^{-i/x}\right|\left|F_{n,h,a}(Nx)-{}_1F_1\left[h,a;2i/x\right]\right|\right\}.\nonumber
\end{align}
We show that the bracket $\{.\}$ is bounded uniformly by $\frac{1}{N}C(x_0,s)$. In the following, we look at the three summands in the bracket separately. For the first one, we have by \eqref{bound on convergence of G} and \eqref{bound on convergence of F} that
\begin{equation*}
\left|\left(1-i/(Nx)\right)^{(N-s)/2-(N-n)}\left(1+i/(Nx)\right)^{-(N+\overline{s})/2}F_{n,h,a}(Nx)\right|\leq C(x_0,s).
\end{equation*}
Moreover, it is easy to check (for example, by using Stirling formula) that
\begin{equation*}
\left| \frac{\Gamma(2\Re{s}+N+1)}{N^{2\Re{s}+1}\Gamma(N)} - 1 \right| \leq \frac{1}{N} C(s).
\end{equation*}
%But it is well known that
%\begin{equation}\label{exp converges like 1/N}
%\left|\frac{(1+(1+2\Re{s})/N)^{N+}{e^{2\Re{s}+1}}-1\right|\leq\frac{1}{N}C(s).
%\end{equation}
%This can be seen by writing out the binomial series for $(1+(1+2\Re{s})/N)^N$ and the Taylor series for the exponential function. Then, for $N$ large enough, the binomial series is absolutely convergent and one can interchange summands to obtain the above bound. Using again binomial series,
%\begin{equation*}
%\left|(1+(1+2\Re{s})/N)^{2\Re{s}+1/2}-1\right|\leq\sum_{k=1}^\infty\left|{2\Re{s}+1/2 \choose k}\right|\left|(1+2\Re{s})/N\right|^k\leq \frac{1}{N}C(s).
%\end{equation*}
Now, if some sequence $a_N>0$ converges to $a>0$ in the order $1/N$ as $N\rightarrow\infty$, $\sqrt{a_N}\rightarrow\sqrt{a}$, in the order $1/N$ as well, for $N\rightarrow\infty$. Hence,
\begin{equation*}
\left|D'(N,s)-1\right|=\left|\left(\frac{\Gamma(2\Re{s}+N+1)}{N^{2\Re{s}+1}\Gamma(N)}\right)^{1/2}-1\right|\leq\frac{1}{N}C(s).
\end{equation*}
Thus, the first term in the bracket $\{.\}$ of \eqref{general domination for Phi} is  bounded by $C(x_0,s)/N$.
Let us look at the second term:
\begin{equation*}
\left|F_{n,h,a}(Nx)\right|\leq C(x_0,s),
\end{equation*}
again according to \eqref{bound on convergence of F}. Moreover,
\begin{align}\label{bound for the second term in the famous bracket}
&\left|(1-i/(Nx))^{(N-s)/2-(N-n)}(1+i/Nx)^{-(N+\overline{s})/2}-e^{-i/x}\right|\\
&\leq \left|(1-i/(Nx))^{(N-s)/2}(1+i/(Nx))^{-(N+\overline{s})/2}-e^{-i/x}\right|\left|(1-i/(Nx))^{-(N-n)}\right|\nonumber\\
&+\left|e^{-i/x}\right|\left|(1-i/(Nx))^{-(N-n)}-1\right|.\nonumber
\end{align}
It is clear, that the second term in the sum is bounded by $C(x_0)/N$. For the first term, the second factor is bounded by $C(x_0)$, whereas for the first factor, we have the following:
\begin{align}\label{second term bound}
&\left|\left(\frac{1-i/(Nx)}{1+i/(Nx)}\right)^{N/2}\left(\frac{1-i/(Nx)}{1+i/(Nx)}\right)^{-i\Im{s}/2}\left(1+1/(Nx)^2\right)^{-\Re{s}/2}-e^{-i/x}\right|\\
\leq& \left|\left(\frac{1-i/(Nx)}{1+i/(Nx)}\right)^{N/2}-e^{-i/x}\right|\left|\left(\frac{1-i/(Nx)}{1+i/(Nx)}\right)^{-i\Im{s}/2}\right|\left|\left(1+1/(Nx)^2\right)^{-\Re{s}/2}\right|\nonumber\\
+&\left|e^{-i/x}\right|\left|\left(\frac{1-i/(Nx)}{1+i/(Nx)}\right)^{-i\Im{s}/2}-1\right|\left|\left(1+1/(Nx)^2\right)^{-\Re{s}/2}\right|\nonumber\\
+&\left|e^{-i/x}\right|\left|\left(1+1/(Nx)^2\right)^{-\Re{s}/2}-1\right|.\nonumber
\end{align}
We investigate all terms in this sum separately: $|(1+1/(Nx)^2)^{-\Re{s}/2}-1|$ can be bounded by $C(x_0,s)/N$ using binomial series, and
\begin{equation*}
\left|\left(\frac{1-i/(Nx)}{1+i/(Nx)}\right)^{-i\Im{s}/2}\right|=\left|\exp\{-\Im{s}\textrm{Arg}(1+i/Nx)\}\right|\leq C(x_0,s).
\end{equation*}
Furthermore,
\begin{align*}
&\left|\left(\frac{1-i/(Nx)}{1+i/(Nx)}\right)^{-i\Im{s}/2}-1\right|=\left|\exp\{-\Im{s}\textrm{Arg}(1+i/(Nx))\}-1\right|\\
=&\left|\exp\{-\Im{s}\textrm{Arctan}(1/(Nx))\}-1\right|\leq\left|\sum_{k=0}^\infty\frac{\left(-\Im{s}\sum_{n=0}^\infty\frac{(-1)^n}{2n+1}\left(1/(Nx)\right)^{2n+1}\right)^k}{k!}-1\right|\\
\leq&\frac{1}{N}C(x_0,s).
\end{align*}
Here, we use the fact that the Taylor series for the arctangent is absolutely convergent if $0<1/(Nx)<1$, which is true for $N$ large enough. Now, by considering the series of the complex logarithm of $1\pm i/(Nx)$ (absolutely convergent for $N$ large enough), one can show that
\begin{equation*}
\left|\left(1\pm i/(Nx)\right)^{\mp N/2}-e^{-i/(2x)}\right|\leq\frac{1}{N}C(x_0).
\end{equation*}
The remaining terms in the sum \eqref{second term bound} are clearly bounded by $C(x_0,s)$ and hence, the second term in the sum \eqref{general domination for Phi} converges to zero in the order $1/N$.

We investigate the third term in \eqref{general domination for Phi}: Clearly, $\left|e^{-i/x}\right|=1$. The second factor in the third term requires somewhat more work:
\begin{align*}
&\left|F_{n,h,a}(Nx)-{}_1F_1[h,a;2i/x]\right|\\
=&\left|\sum_{k=0}^\infty\frac{(-n)_k(h)_k2^k}{(a)_kk!}\left(\frac{1}{1+iNx}\right)^k-\sum_{k=0}^\infty\frac{(h)_k(2i)^k}{(a)_kk!}\left(\frac{1}{x}\right)^k\right|\\
\leq&\sum_{k=1}^\infty\frac{(|h|)_k2^k}{(a)_kk!}\left|(-n)_k\left(\frac{1}{i-Nx}\right)^k-\left(\frac{1}{x}\right)^k\right|,
\end{align*}
where the last inequality is true because of the absolute convergence of both sums. Now,
\begin{align*}
& \left| (-n)_k \left(\frac{1}{i-Nx} \right)^k - \left( \frac{1}{x} \right)^k \right| \\
\leq &  \frac{1}{x_0^k} \, \left| 1 - \frac{(-n)_k}{((i/x)-N)^k} \right|\\
= & \frac{1}{x_0^k} \, \left| 1 - \prod_{l=N-n}^{N-n+k-1} \, \frac{l-N}{(i/x) - N} \right|\\
= & \frac{1}{x_0^k} \, \left| 1 - \prod_{l=N-n}^{N-n+k-1} \, \frac{(N-l)_+}{N - (i/x)} \right|.
\end{align*}
Since all the factors in the last product have a module smaller than 1, it is possible to
deduce:
\begin{align*}
& \left| (-n)_k \left(\frac{1}{i-Nx} \right)^k - \left( \frac{1}{x} \right)^k \right| \\
\leq &  \frac{1}{x_0^k} \, \sum_{l=N-n}^{N-n+k-1} \, \left| 1 - \frac{(N-l)_+}{N-(i/x)} \right|\\
\leq &  \frac{1}{x_0^k}  \, \sum_{l=N-n}^{N-n+k-1} \, \frac{l + 1/x }{N} \\
\leq & \frac{1}{x_0^k} \, \frac{k^2 + k/x_0}{N}.
\end{align*}
%
%\begin{align*}
%\sum_{k=1}^\infty\frac{(|s|)_k2^k}{(2\Re{s}+1)_kk!}\left|\frac{(-N)\cdots(-N+k-1)}{\sum_{l=0}^k{k\choose l}i^lx^{k-l}(-N)^{k-l}}-\left(\frac{1}{x}\right)^k\right|,\quad\textrm{ if }n=N,\\
%\sum_{k=1}^\infty\frac{(|s+1|)_k2^k}{(2\Re{s}+2)_kk!}\left|\frac{(-N+1)\cdots(-N+k)}{\sum_{l=0}^k{k\choose l}i^lx^{k-l}(-N)^{k-l}}-\left(\frac{1}{x}\right)^k\right|,\quad\textrm{ if }n=N-1.
%\end{align*}
%Both terms are bounded by $\frac{1}{N}C(x_0,s)$. To see this, note that $(-n)_k=0$, if $k$ is large and apply the ratio test to get the absolute convergence of the last sum. This proves that the third term in \eqref{general domination for Phi} is bounded by $C(x_0,s)/N$.
This bound implies easily that:
$$\left|F_{n,h,a}(Nx)-{}_1F_1[h,a;2i/x]\right| \leq \frac{C(s,x_0)}{N},$$ and we can deduce:
\begin{equation*}
|\Phi_N(x)-\Phi(x)|\leq\frac{1}{N}\frac{C(x_0,s)}{x^{\Re{s}+(N-n)}}.
\end{equation*}
Therefore, \eqref{P converges like 1/N} and \eqref{Q converges like 1/N} are proved for $p=0$.

It remains to prove that
\begin{equation*}
|\Phi_N'(x)-\Phi'(x)|\leq\frac{1}{N}\frac{C(x_0,s)}{x^{\Re{s}+(N-n)+1}},
\end{equation*}
to show \eqref{P converges like 1/N} and \eqref{Q converges like 1/N} for $p=1$. But this is immediate using the same methods as above and the fact that we can write
\begin{align*}
&\Phi_N'(x)=\\
&D'(N,s)e^{\pi\Im{s}/2}\left(\frac{2}{x}\right)^{\Re{s}+(N-n)}\left(1-i/(Nx)\right)^{(N-s)/2-(N-n)}\left(1+i/(Nx)\right)^{-(\overline{s}+N)/2}\\
&\cdot\left[\frac{-(\Re{s}+(N-n))}{x}F_{n,h,a}(Nx)\right.\\
&+\frac{i}{x^2}\left\{\left(\frac{1-s/N}{2}-\frac{N-n}{N}\right)\frac{1}{1-i/(Nx)}+\frac{1+\overline{s}/N}{2}\frac{1}{1+i/(Nx)}\right\}F_{n,h,a}(Nx)\\
&\left.+\sum_{k=0}^\infty\frac{(-n)_k(h)_kk2^{k+1}}{(a)_kk!}\left(-\frac{iN}{2}\right)\left(\frac{1}{1+iNx}\right)^{k+1}\right],
\end{align*}
and
\begin{align*}
&\Phi'(x)=\\
&e^{\pi\Im{s}/2}\left(\frac{2}{x}\right)^{\Re{s}+(N-n)}e^{-i/x}\\
&\cdot\left[\frac{-(\Re{s}+(N-n))}{x}{}_1F_1[h,a;2i/x]+\frac{i}{x^2}{}_1F_1[h,a;2i/x]\right.\\
&\left.-\sum_{k=0}^\infty\frac{(h)_k(2i)^kk}{(a)_kk!}\left(\frac{1}{x}\right)^{k+1}\right].
\end{align*}
This ends the proof.
\end{proof}

Now we prove Theorem \ref{rate of convergence for large x}.
Let us first prove the following result: for all $n \in \mathds{N}^*$, and for all symmetric and positive $n \times n$ matrices
$A$ and $B$ such that $\sup_{1 \leq i,j \leq n} |A_{i,j}| \leq \alpha$,
$\sup_{1 \leq i,j \leq n} |B_{i,j}| \leq \alpha$ and $\sup_{1 \leq i,j \leq n} |A_{i,j}- B_{i,j}| \leq \beta$ for
some $\alpha, \beta > 0$, one has
\begin{equation}
|\det(B) - \det(A)| \leq \beta n^2 \alpha^{n-1}. \label{A-B}
\end{equation}
Indeed, the following formula holds:
$$\det(B) - \det(A) = \int_0^1 \, d \lambda \, \operatorname{Diff} \, \det [A + \lambda(B-A)].(B-A)$$
where for $C:= A + \lambda(B-A)$, $\operatorname{Diff} \, \det  [C].(B-A)$ denotes the image of the matrix $B-A$
by the differential of the deteminant, taken at point $C$. Now, $C$ is symmetric, positive, and $|C_{i,j}|
\leq \alpha$ for all indices $i,j$, since $C$ is
a barycenter of $A$ and $B$, with positive coefficients. Moreover, the derivative of $C$ with respect to
the coefficent of indices $i,j$ is (up to a possible change of sign) the determinant of the $(n-1) \times
(n-1)$ matrix obtained
by removing the line $i$ and the column $j$ of $C$. By using the same
arguments as in the proof of inequality \eqref{bound on the determinant (in proof of key lemma)}, one can
easily deduce that this derivative is bounded by $\alpha^{n-1}$. Hence:
$$|\det(B) - \det(A)| \leq \int_0^1 \, d \lambda \, \alpha^{n-1} \sum_{1 \leq i,j \leq n} |B_{i,j} - A_{i,j}|$$
which imples \eqref{A-B}. Now, we can compare the determinants
of $(K_{[N]}(x_i,x_j))_{i,j=1}^n$ and $(K_\infty(x_i,x_j))_{i,j=1}^n$ for $x_1,\ldots,x_n>x_0$
by applying \eqref{A-B} to:
$$A_{i,j} = (x_i x_j)^{\Re(s) + 1} \, K_{[N]} (x_i,x_j),$$
$$B_{i,j} = (x_i x_j)^{\Re(s) + 1} \, K_{\infty} (x_i,x_j),$$
$$\alpha = C(x_0,s), \; \beta = C(x_0,s)/N.$$
Here, we use the bounds for $K_{[N]}$, $K_\infty$ and their difference given in Proposition \ref{convergence of K} and in Lemma \ref{rate of convergence of K_N}.
We obtain:
\begin{align*}
&\left|\det(K_{[N]}(x_i,x_j)_{i,j=1}^n)-\det(K_\infty(x_i,x_j)_{i,j=1}^n)\right|\\
\leq& \frac{1}{(x_1 \cdots x_n)^{2\Re(s)+2}} \, \frac{n^2}{N} \, (C(x_0,s))^n.
\end{align*}
This implies
\begin{align*}
&\left|P\left[\frac{\lambda_1(N)}{N}\leq x\right]-\det(I-K_\infty)|_{L_2(t,\infty)}\right|\\
\leq&\sum_{n=1}^\infty\frac{1}{n!}\int_{(x,\infty)^n}\left|\det(K_{[N]}(x_i,x_j)_{i,j=1}^n)-\det(K_\infty(x_i,x_j)_{i,j=1}^n)\right|dx_1\cdots dx_n\\
\leq&\sum_{n=1}^\infty\frac{1}{n!} \, \frac{n^2}{N} \left(\int_{(x,\infty)}\frac{C(x_0,s)}{y^{2\Re{s}+2}}dy\right)^n \\
\leq& \frac{1}{N} \sum_{n=1}^\infty\frac{n}{(n-1)!} \left(\int_{(x_0,\infty)}\frac{C(x_0,s)}{y^{2\Re{s}+2}}dy\right)^n
\leq C(x_0,s)/N,
\end{align*}
since the last sum is convergent and depends only on $x_0$ and $s$.

\section{Concluding remark about $U(N)$}\label{Translation to UN}

With the notations and results from Subsection \ref{Scaling Limits}, we know  that the distribution of $\lambda_1(N)$, the largest eigenvalue of a matrix in $H(N)$ under the distribution \eqref{Full Hua-Pickrell}, can be written as
\begin{equation}\label{largest ev distribution in HN}
P[\lambda_1(N)\leq a]=\exp\left(-\int_a^\infty\frac{\sigma(t)}{1+t^2}dt\right).
\end{equation}

Using the Cayley transform $H(N)\ni X\mapsto U=\frac{X+i}{X-i}\in U(N)$, we can map the generalized Cauchy measure from $H(N)$ to the measure \eqref{Hua-Pickrell on U} on $U(N)$. The inverse of the Cayley transform writes as
\begin{equation*}
\theta\longmapsto i\frac{e^{i\theta}+1}{e^{i\theta}-1}=\cot\left(\frac{\theta}{2}\right),
\end{equation*}
for $\theta\in[-\pi,\pi]$. $\theta=0$ is mapped to $\infty$ by definition. Using this application, equation (\ref{largest ev distribution in HN}) turns into:
\begin{equation}\label{largest ev distribution in UN}
P[\theta_1(N)\geq y]=\exp\left(-\frac{1}{2}\int_0^yd\phi\;\sigma\left(\cot\left(\frac{\phi}{2}\right)\right)\right),
\end{equation}
for $y=2\arccot(a)$, $y\in[0,2\pi]$, and $e^{i\theta_1(N)}=\frac{\lambda_1(N)+i}{\lambda_1(N)-i}$. $\theta_1(N)$ being here in $[0,2\pi]$ (and not in $[-\pi,\pi]$!). In other words, the distribution of the largest eigenvalue on the real line of a random matrix $H\in H(N)$ with measure \eqref{Full Hua-Pickrell}, maps to the distribution of the eigenvalue with smallest angle of a random matrix $U\in U(N)$ satisfying the law \eqref{Hua-Pickrell on U}. Here, smallest angle has to be understood as the eigenvalue which is closest to 1 looking counterclockwise on the circle from the point 1.

According to \cite{Bourgade-Nikeghbali-Rouault}, the eigenvalues $\{e^{i\theta_1},\ldots,e^{i\theta_N}\}$, (recall that $\theta_i\in[-\pi,\pi]$) of a random unitary matrix $U$, satisfying the law \eqref{Hua-Pickrell on U}, also determine a determinantal point process with correlation kernel
\begin{align}\label{correlation kernel on UN}
&K_N^U(e^{i\alpha},e^{i\beta})\\
&=d_N(s)\sqrt{w_U(\alpha)w_U(\beta)}\frac{e^{iN\frac{\alpha-\beta}{2}}Q_N^s(e^{-i\alpha})
Q_N^{\overline{s}}(e^{i\beta})-e^{-iN\frac{\alpha-\beta}{2}}Q_N^{\overline{s}}(e^{i\alpha})Q_N^s(e^{-i\beta})}
{e^{i\frac{\alpha-\beta}{2}}-e^{-i\frac{\alpha-\beta}{2}}},\nonumber
\end{align}
where $d_N(s)=\frac{1}{2\pi}\frac{(\overline{s}+1)_N(s+1)_N}{(2\Re{s}+1)_N N!}\frac{\Gamma(1+s)\Gamma(1+\overline{s})}{\Gamma(1+2\Re{s})}$, $Q_N^s(x)=\hg[s,-n,-n-\overline{s};x]$ and $w_U$ is the weight defined after (\ref{Hua-Pickrell on U}). If $N\rightarrow\infty$, the rescaled correlation kernel $\frac{1}{N}K_N^U(e^{i\alpha/N},e^{i\beta/N})$ converges to
\begin{align}\label{correlation kernel on Uinfinity}
&K^U(\alpha,\beta)\\
&=e(s)|\alpha\beta|^{\Re{s}}e^{-\frac{\pi}{2}\Im{s}(\textrm{Sgn}(\alpha)+\textrm{Sgn}(\beta))}\frac{e^{i\frac{\alpha-\beta}{2}}Q^s(-i\alpha)Q^{\overline{s}}(i\beta)-e^{-i\frac{\alpha-\beta}{2}}Q^{\overline{s}}(i\alpha)Q^s(-i\beta)}{\alpha-\beta},\nonumber
\end{align}
where $e(s)=\frac{1}{2\pi i}\frac{\Gamma(s+1)\Gamma(\overline{s}+1)}{\Gamma(2\Re{s}+1)^2}$, and $Q^s(x)={}_1F_1[s,2\Re{s}+1;x]$ (again according to \cite{Bourgade-Nikeghbali-Rouault}). In \cite{Bourgade-Nikeghbali-Rouault}, it is also shown that the kernel $K^U$ coincides up to multiplication by a constant with the limiting kernel $K_\infty$ from \eqref{definition of K infinity} if one changes the variables in (\ref{correlation kernel on Uinfinity}) to  $\alpha=\frac{2}{x}$ and $\beta=\frac{2}{y}$, $x,y\in\mathds{R}^\ast$. This not surprising because a scaling $x\mapsto Nx$ for the eigenvalues in the Hermitian case corresponds to a scaling $\alpha\mapsto\frac{\alpha}{N}$ for the eigenvalues in the unitary case as can be seen from the elementary fact that for $x\in\mathds{R}^*$, and $N\in\mathds{N}$, one has
\begin{equation}\label{transformation of the scaling by cayley, formula}
\frac{Nx+i}{Nx-i}=e^{\frac{2i}{Nx}+O(N^{-2})}.
\end{equation}

\begin{remark}
Note that because of the $O(N^{-2})$ term in the argument of (\ref{transformation of the scaling by cayley, formula}), it is not possible to give an identity involving the kernel $K_N$ of Theorem \ref{correlation Kernel} and the kernel (\ref{correlation kernel on UN}).
\end{remark}

\bibliographystyle{amsplain} 

\begin{thebibliography}{10}

\bibitem{Adler-Forrester} M. Adler, P. Forrester, T. Nagao, P. van Moerbeke, \textit{Classical skew orthogonal polynomials and random matrices}, J. Stat. Phys., vol. 99, \textbf{1-2}, 2000, pp. 141--170.

%\bibitem{Andrews-Askey-Roy} G.E. Andrews, R. Askey, R. Roy, \textit{Special Functions}, Encyclopedia of Mathematics and its Applications, vol. 71,Cambridge University Press, 1999.

\bibitem{Borodin-Deift} A. Borodin, P. Deift, \textit{Fredholm determinants, Jimbo-Miwa-Ueno-$\tau$-functions, and representation theory}, Commun. Pure Appl. Math., vol. 55, 2002, pp. 1160--1230.

\bibitem{Borodin-Olshanski} A. Borodin, G. Olshanski, \textit{Infinite random matrices and Ergodic
measures}, Commun. Math. Phys., vol. 223,
\textbf{1}, 2001, pp. 87--123.

\bibitem{Bourgade-Nikeghbali-Rouault} P. Bourgade, A. Nikeghbali,
A. Rouault, \textit{Hua-Pickrell measures on general compact
groups}, \href{http://arxiv.org/abs/0712.0848}{arXiv:0712.0848v1},
2007.

\bibitem{Bourgade-Nikeghbali-Rouault2} P. Bourgade, A. Nikeghbali,
A. Rouault, \textit{Circular Jacobi ensembles and deformed Verblunsky coefficients}, \href{http://arxiv.org/abs/0804.4512}{arXiv:0804.4512v2},
2008.

%\bibitem{Bowick-Brezis} M.J. Bowick, E. Br?zin, \textit{Universal scaling of the tail of the density of eigenvalues in random matrix models}, Phys. Letts., B268, 1991, pp. 21--28.

\bibitem{Choup} L. Choup, \textit{Edgeworth expansion of the largest eigenvalue distribution function of GUE and LUE}, IMRN, vol. 2006, 2006, pp. 61049.

\bibitem{Cosgrove-Scoufis} C.M. Cosgrove, G. Scoufis,
\textit{Painlev? classification of a class of differential equations
of the second order and second degree}, Stud. Appl. Math., vol. 88, \textbf{1}, 1993, pp. 25--87.

\bibitem{El-Karouie} N. El Karouie,
\textit{A rate of convergence result for the largest eigenvalue of complex white Wishart matrices}, Ann. Probab., vol. 34, \textbf{6}, 2006, pp. 2077--2117.

%\bibitem{Forrester2} P.J. Forrester, \textit{The spectrum edge of random matrix ensembles}, Nuclear Physics, B402, 1993, pp. 709--728.

\bibitem{Forrester} P.J. Forrester, \textit{Random Matrices and Log
Gases}, Book in preparation.



\bibitem{Forrester-Witte2} P.J. Forrester, N.S. Witte, \textit{Application of the $\tau$-function theory of Painlev? equations to random matrices: $P_{VI}$, the JUE, CyUE, cJUE and scaled limits}, Nagoya Math. J., vol. 174, 2002, pp. 29--114.

\bibitem{Forrester-Witte3} P.J. Forrester, N.S. Witte, \textit{Random matrix theory and the sixth Painlev\'e equation}, J. Phys. A: Math. Gen., vol. 39, 2006, pp. 12211--12233.

\bibitem{Garoni-Forrester-Frankel} T.M. Garoni, P.J. Forrester, N.E. Frankel, \textit{Asymptotic corrections to the eigenvalue density of the GUE and LUE}, J. Math. Phys., vol. 46, 2005, pp. 103301.

%\bibitem{Hua} L.K. Hua, \textit{Harmonic analysis of functions of several complex variables in the classical domains}, Chinese edition: Science Press, Peking, 1958, Russian edition: IL, Moscow, 1959, English edition: Transl. Math. Monographs, vol. 6, AMS, 1963.

\bibitem{JMMS} M. Jimbo, T. Miwa, Y. M?ri, M. Sato, \textit{Density matrix of an impenetrable Bose gas and the fifth Painlev? transcendent}, Physica 1D, 1980, pp. 80--158.

\bibitem{Johnstone} I.M. Johnstone, \textit{On the distribution of the largest principal component}, Ann. Math. Stat., vol. 29, 2001, pp. 295--327.

\bibitem{Kamien-Politzer} R.D. Kamien, H.D. Politzer, M.B. Wise, \textit{Universality of random-matrix predictions for the statistics of energy levels}, Phys. Rev. Letts., vol. 60, 1988, pp. 1995--1998.

%\bibitem{Keating-Snaith1} J.P. Keating, N.C. Snaith, \textit{Random Matrix Theory and $\zeta(1/2+it)$}, Communications in Mathematical Physics, vol. 214, 2000, pp. 57--89.

%\bibitem{Keating-Snaith2} J.P. Keating, N.C. Snaith, \textit{Random matrix theory and L-functions at $s=1/2$}, Communications in Mathematical Physics, vol. 214, 2000, pp. 91--110.

\bibitem{Mahoux-Mehta} G. Mahoux, M.L. Mehta, \textit{A method of integration over matrix variables: IV}, J. Phys. I France, vol. 1, 1991, pp. 1093--1108.

\bibitem{Mehta} M.L. Mehta, \textit{Random Matrices}, Academic
Press, San Diego, 1991.

\bibitem{Nagao-Wadati} T. Nagao, M. Wadati, \textit{Correlation functions of random matrix ensembles related to classical orthogonal polynomials}, J. Phys. Soc. Japan, vol. 6, 1991, pp. 3298--3322.

\bibitem{Pastur} L.A. Pastur, \textit{On the universality of the level spacing distribution for some ensembles of random matrices}, Letts. Math. Phys., vol. 25, 1992, pp. 259--265.

\bibitem{peche} S. P\'ech\'e, \textit{The largest eigenvalues of small rank perturbations of Hermitian random matrices}, Prob. Theor. Rel. Fields, vol. 134, \textbf{1}, 2006, pp. 127--174.

%\bibitem{Pickrell1} D. Pickrell, \textit{Measures on infinite-dimensional Grassmann manifolds}, J. Func. Anal., vol. 70, 1987, pp. 323--356.

%\bibitem{Pickrell2} D. Pickrell, \textit{Mackey analysis of infinite classical motion groups}, Pacific J. Math., vol. 150, 1991, pp. 139--166.

\bibitem{soshnikov2} A. Soshnikov, \textit{Universality at the edge of the spectrum in Wigner random matrices}, Commun.  Math. Phys., vol. 207 , 1999, pp. 697--733.

\bibitem{soshnikov} A. Soshnikov, \textit{Poisson statistics for the largest eigenvalues in random matrix ensembles. Mathematical physics of quantum mechanics}, Lecture Notes in Phys., vol. 690, Springer, Berlin, 2006.
%\bibitem{Soshnikov} A. Soshnikov, \textit{Determinantal random
%point fields}, Russian Math. Surveys, vol. 55, 2000, pp. 923--975.

%\bibitem{Szego} G. Szeg?, \textit{Orthogonal Polynomials}, AMS
%Colloquium Publications, vol. 23, New York, Amer.Math.Soc., 1959.

\bibitem{Tracy-Widom2} C.A. Tracy, H. Widom, \textit{Level-Spacing distributions and the Airy kernel}, Commun.  Math. Phys., vol. 159, 1994, pp. 151--174.

\bibitem{Tracy-Widom3} C.A. Tracy, H. Widom, \textit{Level-Spacing distributions and the Bessel kernel}, Commun.  Math. Phys., vol. 161, 1994, pp. 289--309.

\bibitem{Tracy-Widom} C.A. Tracy, H. Widom, \textit{Fredholm determinants, differential equations and matrix
models}, Commun.  Math. Phys., vol. 163,
\textbf{1}, 1994, pp. 33--72.

\bibitem{Forrester-Witte}  N.S. Witte, P.J. Forrester,  \textit{Gap probabilities in the finite and scaled Cauchy
random matrix ensembles}, Nonlinearity, vol. 13, 2000, pp. 1965--1986.



\end{thebibliography}

\end{document}